\documentclass[11pt,reqno]{amsart}
\usepackage[T1]{fontenc}
\usepackage{lmodern}
\usepackage[a4paper,margin=32mm]{geometry}
\usepackage{amsmath,amssymb,amsthm,mathtools,bm,mathrsfs}
\usepackage{microtype}
\usepackage{xcolor}
\usepackage{enumitem}
\usepackage[colorlinks=true,linkcolor=blue!55!black,citecolor=blue!55!black,urlcolor=blue!55!black]{hyperref}
\usepackage[nameinlink,noabbrev]{cleveref}
\numberwithin{equation}{section}
\newtheorem{theorem}{Theorem}[section]
\newtheorem{proposition}[theorem]{Proposition}
\newtheorem{lemma}[theorem]{Lemma}
\newtheorem{corollary}[theorem]{Corollary}
\theoremstyle{definition}

\theoremstyle{remark}
\newtheorem{remark}[theorem]{Remark}

\newcommand{\kai}{\boldsymbol{1}}
\newcommand{\R}{\mathbb R}
\newcommand{\Sn}{\mathbb S^{n-1}}

\newcommand{\abs}[1]{\lvert #1\rvert}
\newcommand{\Per}{\operatorname{Per}}
\newcommand{\diam}{\operatorname{diam}}

\newcommand{\Sph}{\mathbb S^{n-1}}
\newcommand{\cH}{\mathcal H}
\newcommand{\cJ}{\mathcal J}
\newcommand{\cM}{\mathcal M}

\newcommand{\eps}{\varepsilon}

\allowdisplaybreaks

\title[A Brunn--Minkowski theory for random pair distances]
{A Brunn--Minkowski Theory for the Distribution of Random Pairs of Points in Convex Bodies}
       \author[S. Wang]{Sinan Wang}
\address{Institute of Mathematics,
		Hunan University,  Changsha,  410082,  China}
	\email{wangsinan@hnu.edu.cn}
\date{\today}

\begin{document}
\begin{abstract}
Let $K\subset\R^n$ be a convex body, let $X$ and $Y$ be independent uniform points of $K$, and set $R_K=(|K|/\omega_n)^{1/n}$.  We consider the distribution function
\[
 D_K(\rho)=\mathbb P\{\abs{X-Y}/R_K<\rho\},
\]
and its homogeneous counterpart
\[
 \mathcal J_\rho(K)=\iint_{K\times K}\kai_{\{\abs{x-y}<\rho R_K\}}\,dx\,dy.
\]
Thus $D_K(\rho)$ is the proportion of ordered pairs of points of $K$ whose distance is less than $\rho R_K$, or equivalently the radial distribution function of the normalized covariogram of $K$.

The functional satisfies a Brunn--Minkowski inequality, with equality precisely for homothetic bodies, and its differential gives a first Minkowski inequality.  We characterize the Borel measures that occur as Wulff differentials: they are exactly the nonzero finite positive measures with zero centroid which are not concentrated on a great subsphere, and the corresponding body is unique up to translation.  We further prove fixed-volume and stationary rigidity of the ball in the subcritical range, together with regularity for the resulting equation.  At saturation, the equation is the classical Minkowski equation.
\end{abstract}
\maketitle

\section{Introduction}

The Brunn--Minkowski theory studies geometric functionals of convex bodies and the measures arising as their differentials.  One of its basic problems is to decide which measures arise in this way.  The functional studied here comes from geometric probability.  Let $K\subset\R^n$ be a convex body, and let $X$ and $Y$ be independent uniform points of $K$.  Distances are normalized by the radius of the Euclidean ball with the same volume as convex body $K$:
\[
 R_K:=\left(\frac{|K|}{\omega_n}\right)^{1/n},
 \qquad D_K(\rho):=\mathbb P\left\{\frac{\abs{X-Y}}{R_K}<\rho\right\},\qquad \rho>0.
\]
The function $D_K$ is invariant under translations and dilations.  At a given $\rho$, it is the proportion of ordered pairs of points in $K$ whose distance is less than $\rho R_K$.  Random chords and distance distributions belong to the classical subject of integral geometry; see Santal\'o \cite{Santalo2004} and Schneider--Weil \cite{SchneiderWeil2008}.  The normalization by $R_K$ compares different bodies at one relative scale.

We work with the homogeneous form of this distribution,
\begin{equation}
 \mathcal J_\rho(K):=|K|^2D_K(\rho)
 =\iint_{K\times K}\kai_{\{\abs{x-y}<\rho R_K\}}\,dx\,dy.
\label{eq:intro-energy}
\end{equation}
The functional $\mathcal J_\rho$ is homogeneous of degree $2n$, the homogeneity compatible with Minkowski addition and Wulff variation.  We study its differential, the associated Minkowski problem, and the corresponding Brunn--Minkowski, rigidity, and regularity questions.

Let
\begin{equation}
 g_K(z):=|K\cap(K+z)|,\qquad z\in\R^n.
\label{eq:intro-covariogram}
\end{equation}
Then $|K|^{-2}g_K(z)\,dz$ is the law of $X-Y$, and consequently
\begin{equation}
 D_K(\rho)=\frac1{|K|^2}\int_{B_{\rho R_K}}g_K(z)\,dz.
\label{eq:intro-displacement-profile}
\end{equation}
The covariogram occurs in random-set theory, geometric probability, chord-distribution questions, and phase retrieval; see Matheron \cite{Matheron1975}, Santal\'o \cite{Santalo2004}, Schneider--Weil \cite{SchneiderWeil2008}, and Bianchi's survey \cite{Bianchi2023}.  Formula~\eqref{eq:intro-displacement-profile} expresses $D_K$ as the radial distribution function of the normalized displacement $X-Y$.  For the corresponding random-chord distribution without the normalization by $R_K$, see Ren \cite{Ren2014}.

The covariogram also occurs in affine convex geometry.  The superlevel sets
\[
 C_t(K):=\{z\in\R^n:g_K(z)\ge t\},\qquad 0<t<|K|,
\]
are the convolution bodies of $K$.  Their behavior as $t\uparrow|K|$ yields a covariogram description of affine surface area; see Gardner \cite{Gardner2006}, notes 9.3.  The layer-cake formula gives
\[
 \mathcal J_\rho(K)=\int_{B_{\rho R_K}}g_K(z)\,dz
 =\int_0^{|K|}|B_{\rho R_K}\cap C_t(K)|\,dt.
\]
Thus $\mathcal J_\rho$ is a Euclidean radial average of the convolution bodies.  It is not affine invariant, since the question concerns Euclidean relative distances under Minkowski variation.

The choice $r_K=\rho R_K$ has two simple consequences.  First, the ball used to test the displacement $X-Y$ has fixed relative volume
\begin{equation}
 \frac{|B_{r_K}|}{|K|}=\rho^n,
\label{eq:intro-relative-ball-volume}
\end{equation}
so $\rho$ has the same meaning for every body.  Second,
\begin{equation}
 \mathcal J_\rho(\lambda K)=\lambda^{2n}\mathcal J_\rho(K),
 \qquad D_{\lambda K}(\rho)=D_K(\rho)\qquad(\lambda>0).
\label{eq:intro-scaling}
\end{equation}
Neither assertion holds for a fixed radius.  This is the reason for allowing the radius to depend on the volume.

Fubini's theorem also gives the finite-range representation
\begin{equation}
 \mathcal J_\rho(K)=|K|\,|B_{r_K}|-\Per_{r_K}(K)
 =\rho^n|K|^2-\Per_{r_K}(K),
\label{eq:intro-nonlocal-perimeter}
\end{equation}
where
\[
 \Per_r(K):=\iint_{K\times(\R^n\setminus K)}\kai_{\{\abs{x-y}<r\}}\,dx\,dy.
\]
On a fixed-volume class, maximizing the short-pair probability is equivalent to minimizing a compact-range nonlocal perimeter; compare Cesaroni--Novaga \cite{CesaroniNovaga2018}.  Moreover, as $\rho\downarrow0$, the first correction to the short-pair probability is the dimensionless isoperimetric ratio:
\begin{equation}
 D_K(\rho)=\rho^n-c_n\frac{P(K)}{R_K^{n-1}}\rho^{n+1}+o(\rho^{n+1})
 \qquad(\rho\downarrow0).
\label{eq:intro-small-rho}
\end{equation}
On the other hand, $D_K(\rho)=1$ when $\rho\ge\operatorname{diam}K/R_K$.  Since this ratio is at least $2$, with equality only for balls, $\rho=2$ is the saturation value for the ball of volume $|K|$.

On a fixed-volume class, $r_K$ is constant and \eqref{eq:intro-energy} is a Riesz--Sobolev functional.  The Riesz rearrangement inequality and its equality cases identify the ball as the maximizer in the nondegenerate range; see \cite{BLL74, Burchard1996}.  Quantitative versions for the ball kernel are due to Christ \cite{Christ2017} and Frank--Lieb \cite{FrankLieb2019}.  Here the volume, and hence the range in \eqref{eq:intro-energy}, changes under Wulff variation.  The derivative of $r_K$ produces the scalar term below.

For a Wulff perturbation $K_t=[h_K+tf]$, the differential of the probability profile is
\begin{equation}
 \left.\frac{d}{dt}\right|_{t=0}D_{K_t}(\rho)
 =\frac1{|K|^2}\left[\int_{\Sn}f\,d\mathcal M_{\rho,K}
 -\frac{2\mathcal J_\rho(K)}{|K|}\int_{\Sn}f\,dS_K\right].
\label{eq:intro-normalized-variation}
\end{equation}
Here $\mathcal M_{\rho,K}$ is defined in Section~\ref{subsec:functional-and-measure}.  The first term is the fixed-range boundary contribution, and the second comes from the variation of $r_K$.  Thus, under volume-preserving perturbations, $\mathcal M_{\rho,K}$ gives the first-order change of $D_K$.

For any convex body and any continuous Wulff perturbation direction, we prove
\begin{equation}
 \left.\frac{d}{dt}\right|_{t=0}\mathcal J_\rho([h_K+tf])
 =\int_{\Sn}f\,d\mathcal M_{\rho,K},
\label{eq:intro-first-variation}
\end{equation}
no smoothness or strict convexity of $K$ is assumed.  This formula
defines the measure $\mathcal M_{\rho,K}$ and is the weak, or
Aleksandrov, starting point for the theory.

The first structural question is whether this new functional has the
same concavity under Minkowski addition as volume.  This is not a
formal consequence of the Brunn--Minkowski inequality in $\mathbb R^n$:
both the two copies of the body and the admissible chord length vary.
The appropriate object is the set of admissible ordered chords.  It is
convex in $\mathbb R^{2n}$, and its volume is exactly
$\mathcal J_\rho(K)$.  This observation gives the following result.

\begin{theorem}[Brunn--Minkowski theorem]
\label{thm:intro-adaptive-BM}
For $K,L\in\mathcal K^n$ and $0\leq\lambda\leq1$,
\begin{equation}
 \mathcal J_\rho((1-\lambda)K+\lambda L)^{1/(2n)}
 \geq (1-\lambda)\mathcal J_\rho(K)^{1/(2n)}
       +\lambda\mathcal J_\rho(L)^{1/(2n)}.
\label{eq:intro-BM}
\end{equation}
If $0<\lambda<1$, equality holds if and only if $K$ and $L$ are
homothetic.
\end{theorem}

The proof represents \(\mathcal J_\rho(K)\) as the \(2n\)-dimensional volume of a convex body in \(\mathbb R^{2n}\). Differentiating \eqref{eq:intro-BM} yields the
first Minkowski inequality
\begin{equation}
 \frac1{2n}\int_{\Sn}h_L\,d\mathcal M_{\rho,K}
 \geq\mathcal J_\rho(K)^{(2n-1)/(2n)}
       \mathcal J_\rho(L)^{1/(2n)},
\label{eq:intro-first-Minkowski}
\end{equation}
again with equality precisely for homothetic bodies.

The associated Minkowski problem is more subtle than the classical
surface area problem.  The measure $\mathcal M_{\rho,K}$ is a weighted
surface area measure, but its weight is a ball-intersection volume and
also contains the global correction caused by the variation of $r_K$.
Thus neither the density nor the scale can be prescribed independently
of the unknown body.  The first-variation formula and the preceding
concavity identify the correct weak formulation and permit a direct
variational treatment.

\begin{theorem}[Weak Minkowski theorem]
\label{thm:intro-weak-Minkowski}
Let $\mu$ be a nonzero finite positive Borel measure on $\Sn$.  There
exists a convex body $K\in\mathcal K^n$ such that
\begin{equation}
 \mathcal M_{\rho,K}=\mu
\label{eq:intro-measure-problem}
\end{equation}
if and only if
\[
 \int_{\Sn}u\,d\mu(u)=0
 \quad\text{and}\quad
 \mu\ \text{is not concentrated on a great subsphere}.
\]
The solution is unique up to translation.
\end{theorem}

Thus the measure equation \eqref{eq:intro-measure-problem} has a
complete weak-solvability theory before any regularity is imposed on
either the body or the datum.

The variational results also have two distinct rigidity consequences.
On a fixed-volume class the radius $r_K$ is fixed, so the problem falls
within the Riesz rearrangement framework.  Its equality theory gives
the global result below; the endpoint $\rho=2$ is handled by the
isodiametric inequality.  On the other hand, stationarity for arbitrary
ambient deformations yields a pointwise ball-intersection condition.
In the subcritical regime this is the uniform-density condition studied
by Bucur and Fragal\`a \cite{BF23}.

\begin{theorem}[Rigidity]
\label{thm:intro-rigidity}
Let $B_K$ be a Euclidean ball with $|B_K|=|K|$.
\begin{enumerate}[label=\textup{(\roman*)},leftmargin=2.2em]
\item If $0<\rho\leq2$, then
\[
 \mathcal J_\rho(K)\leq\mathcal J_\rho(B_K),
\]
and equality holds if and only if $K$ is a Euclidean ball, up to
translation.  For $\rho>2$, the ball is still a maximizer, but
maximizers are not unique.
\item Let $0<\rho<2$.  If a convex body $K$ is stationary for
$\mathcal J_\rho$ under every compactly supported $C^1$
volume-preserving deformation, then $K$ is a Euclidean ball.
\end{enumerate}
\end{theorem}

Part~\textup{(i)} is the fixed-volume, global statement; part~\textup{(ii)}
is its strong Euler--Lagrange counterpart.  We also consider the weaker
Wulff stationarity equation $\mathcal M_{\rho,K}=cS_K$.  Under strict
convexity it likewise implies that $K$ is a ball, whereas without strict
convexity it controls only averages over normal fibres and is not
equivalent to strong stationarity.

We next ask how much regularity follows from the weak equation itself.
The density of the corresponding surface-area measure contains the
nonlocal factor $q_{r_K,K}$ evaluated at a boundary point of the unknown
body.  The first step is therefore not a bootstrap: one must first
compare this factor from above and below with constants depending only on
the body.  This converts the equation, in gnomonic coordinates, into a
uniformly elliptic Aleksandrov Monge--Amp\`ere equation.  The subsequent
bootstrap uses the regularity of the moving ball-intersection volume
along $\partial K$.

\begin{theorem}[Regularity]
\label{thm:intro-regularity}
Suppose that $\mathcal M_{\rho,K}=f\,d\sigma$.
\begin{enumerate}[label=\textup{(\roman*)},leftmargin=2.2em]
\item If $0<f_-\leq f\leq f_+<\infty$, then $K$ is strictly convex and
      $h_K\in C^{1,\gamma}(\Sn)$ for some $\gamma\in(0,1)$.
\item If $f\in C^\alpha(\Sn)$ and $f>0$, where $0<\alpha<1$, then
\begin{equation}
 h_K\in C^{2,\alpha}(\Sn),
 \qquad
 \nabla_{\Sn}^2h_K+h_KI>0\quad\text{on }\Sn.
\label{eq:intro-C2alpha-regularity}
\end{equation}
\item If $f\in C^{k,\alpha}(\Sn)$, $r_K<\diam K$, and
      $\Theta_{r_K}(K)>0$, then
      $h_K\in C^{k+2,\alpha}(\Sn)$.  Here
      $\Theta_{r_K}(K)>0$ is the uniform chord-transversality condition
      defined in Section~4.
\end{enumerate}
When $r_K\geq\diam K$, the equation reduces to the classical Minkowski
equation.
\end{theorem}

The profile $D_K$ is coarser than the full covariogram and does not, in
general, determine $K$.  It records precisely the radial part of the law
of $X-Y$ selected by concentric balls.  The small-range expansion
\eqref{eq:intro-small-rho} recovers perimeter to first order, while the
saturated regime reduces the measure equation to the classical Minkowski
problem.

Section~2 establishes the first-variation formula, proves
\Cref{thm:intro-adaptive-BM}, and solves the weak Minkowski problem in
\Cref{thm:intro-weak-Minkowski}.  Section~3 proves
\Cref{thm:intro-rigidity} and then discusses the relation between strong
and Wulff stationarity.  Section~4 starts from weak solutions, derives
the regularity asserted in \Cref{thm:intro-regularity}, and then obtains
higher regularity under chord transversality.  The content of this paper is part of my PhD thesis \cite{Wang26}.

\section{The adaptive covariogram functional and its Minkowski problem}\label{sec:ACFMP}
Throughout this section $n\ge2$, $\rho>0$, and $\mathcal K^n$ denotes the class of convex bodies in $\R^n$.  We write $\omega_n=|B_1|$, $\Sn$ for the Euclidean unit sphere, and $\sigma=\mathcal H^{n-1}(\Sn)$.  If $K\in\mathcal K^n$, then $h_K$, $S_K$, and $\nu_K$ denote its support function, surface area measure, and outer unit normal.  These are the standard conventions of convex geometry; see Schneider \cite{Schneider2014}.

For a continuous function $g$ on $\Sn$, whenever the intersection below is nonempty, its Wulff shape is
\[
 [g]:=\bigcap_{u\in\Sn}\{x\in\R^n:x\cdot u\le g(u)\}.
\]
If $g$ is positive, $[g]$ contains the origin in its interior.  After translating $K$ so that $0\in\operatorname{int}K$, every $f\in C(\Sn)$ gives, for all sufficiently small $|t|$, the Wulff perturbation
\begin{equation}
 K_t=[h_K+tf].
\label{eq:wulff-perturbation}
\end{equation}
\setcounter{equation}{1}

\subsection{The functional and the associated measure}
\label{subsec:functional-and-measure}

For brevity, write $V_K:=|K|$.
For $r>0$, define the fixed-radius covariogram energy
\begin{equation}
 \mathcal E_r(K)
 :=\iint_{K\times K}\kai_{\{|x-y|<r\}}\,dx\,dy,
 \label{eq:fixed-energy}
\end{equation}
and define the adaptive energy by
\begin{equation}
 \mathcal J_\rho(K):=\mathcal E_{r_K}(K).
 \label{eq:adaptive-energy}
\end{equation}
The fixed-radius energy is introduced first only to separate the variation
of the body from the variation of the interaction range. 

\begin{lemma}
\label{lem:adaptive-Hausdorff-continuity}
Let $\rho>0$, and let $K_j,K\subset\mathbb R^n$ be compact convex sets
such that $K_j\to K$ in the Hausdorff metric.  Then
\[
 \mathcal J_\rho(K_j)\longrightarrow\mathcal J_\rho(K).
\]
\end{lemma}

\begin{proof}
Suppose first that $|K|=0$. Then $|K_j|\to0$, and the elementary bound
\[
 0\leq\mathcal J_\rho(K_j)\leq |K_j|^2
\]
gives
\[
 \mathcal J_\rho(K_j)\longrightarrow0=\mathcal J_\rho(K).
\]
Assume now that $|K|>0$. Then $K$ is a convex body and
\[
 r_j:=\rho R_{K_j}\longrightarrow r:=\rho R_K>0.
\]
Since $K_j\to K$ in the Hausdorff metric, all the sets $K_j$ and $K$ are
contained in a fixed ball $B_R$. Moreover,
\[
 \kai_{K_j}(x)\longrightarrow\kai_K(x)
\]
for every $x\notin\partial K$. Since the boundary of a convex body has
Lebesgue measure zero, this convergence holds for almost every
$x\in\mathbb R^n$.

For almost every $(x,y)\in B_R\times B_R$, we therefore have
\[
 \kai_{K_j}(x)\kai_{K_j}(y)\longrightarrow
 \kai_K(x)\kai_K(y).
\]
In addition,
\[
 \kai_{\{|x-y|<r_j\}}\longrightarrow\kai_{\{|x-y|<r\}}
\]
whenever $|x-y|\ne r$. The exceptional set
\[
 \{(x,y)\in B_R\times B_R:|x-y|=r\}
\]
has $2n$-dimensional Lebesgue measure zero: for each fixed $x$, its
$y$-section is contained in the sphere $\partial B_r(x)$, which has
$n$-dimensional Lebesgue measure zero. Hence, by Fubini's theorem, the
whole exceptional set is null.

It follows that
\[
 \kai_{K_j}(x)\kai_{K_j}(y)\kai_{\{|x-y|<r_j\}}
 \longrightarrow
 \kai_K(x)\kai_K(y)\kai_{\{|x-y|<r\}}
\]
for almost every $(x,y)\in B_R\times B_R$. Since all these integrands are
bounded by $\kai_{B_R\times B_R}$, the dominated convergence theorem yields
\begin{align*}
 \mathcal J_\rho(K_j)
 &=\int_{B_R}\!\int_{B_R}
 \kai_{K_j}(x)\kai_{K_j}(y)\kai_{\{|x-y|<r_j\}}\,dx\,dy\\
 &\longrightarrow\int_{B_R}\!\int_{B_R}
 \kai_K(x)\kai_K(y)\kai_{\{|x-y|<r\}}\,dx\,dy\\
 &=\mathcal J_\rho(K).
\end{align*}
This proves the claim.
\end{proof}

For $x,z\in\mathbb R^n$, set
\begin{equation}
 q_{r,K}(x):=|K\cap B_r(x)|,
 \label{eq:3}
\end{equation}
 \begin{equation}
g_K(z):=|K\cap(K+z)|,
\label{eq:4}
 \end{equation}
 \begin{equation}
 A_K(r):=r^{n-1}\int_{\mathbb S^{n-1}}g_K(ru)\,d\sigma(u).
 \label{eq:5}
\end{equation}
Both $q_{r,K}$ and $g_K$ are continuous.  By Fubini's theorem and polar
coordinates,
\begin{equation}
 \mathcal E_r(K)=\int_{B_r}g_K(z)\,dz,
 \qquad
 \frac{d}{dr}\mathcal E_r(K)=A_K(r)
 \quad\text{for every }r>0.
 \label{eq:radius-derivative}
\end{equation}

The distinction between the adaptive and fixed-range models is already
visible under dilations.  For a fixed $s>0$,
\begin{equation}
 \mathcal E_s(\lambda K)=\lambda^{2n}\mathcal E_{s/\lambda}(K),
 \qquad \lambda>0,
 \label{eq:fixed-energy-scaling}
\end{equation}
whereas \(\mathcal J_\rho(\lambda K)=\lambda^{2n}\mathcal J_\rho(K)\).
Hence a dilation changes the relative observation scale in the former
experiment, but not in the latter.  Thus the profile $D_K$ defined above
is a scale-free point-pair distance profile.

There is also a useful functional-analytic formulation.  Put
\(f_K=\kai_K\), \(\widetilde f_K(x)=f_K(-x)\), and
\(T_s\varphi=\kai_{B_s}*\varphi\).  Then
\(g_K=f_K*\widetilde f_K\), and
\begin{equation}
 \mathcal J_\rho(K)
 =\langle f_K,T_{r_K}f_K\rangle_{L^2(\mathbb R^n)}
 =\langle g_K,\kai_{B_{r_K}}\rangle_{L^2(\mathbb R^n)}.
 \label{eq:convolution-form}
\end{equation}
Thus, on a fixed-volume class, the functional is the quadratic form of a
bounded self-adjoint convolution operator that is invariant under
translations and rotations.  The covariogram is also called the
set-covariance or autocorrelation of a set; see Matheron
\cite{Matheron1975}.  This formulation places the present functional at the
interface of geometric probability, harmonic analysis, and the
Fourier-analytic study of convex bodies; see also Koldobsky
\cite{Koldobsky2005}.

With the convention
\(\widehat h(\xi)=\int_{\mathbb R^n}e^{-ix\cdot\xi}h(x)\,dx\), Plancherel's
theorem and \eqref{eq:convolution-form} give
\begin{equation}
 \mathcal J_\rho(K)
 =(2\pi)^{-n}\int_{\mathbb R^n}
      |\widehat{\kai_K}(\xi)|^2
      \widehat{\kai_{B_{r_K}}}(\xi)\,d\xi,
 \qquad
 \widehat{\kai_{B_s}}(\xi)
 =(2\pi)^{n/2}s^{n/2}|\xi|^{-n/2}J_{n/2}(s|\xi|).
 \label{eq:fourier-form}
\end{equation}
The second identity is understood by continuity at $\xi=0$.  Hence the hard
distance cutoff is a radial Fourier--Bessel multiplier on the power spectrum
\(|\widehat{\kai_K}|^2\).  The multiplier changes sign, so the positivity
of \(\mathcal J_\rho\) is a geometric fact about characteristic functions,
not positivity of the quadratic form on all of $L^2$.

To make the scale-free content explicit, let
\(\overline K=R_K^{-1}K\), so that \(|\overline K|=\omega_n\).  The
covariogram representation and polar coordinates yield
\begin{equation}
 D_K(\rho)
 =\frac{1}{\omega_n^2}\int_{B_\rho}g_{\overline K}(z)\,dz,\quad
 D_K'(\rho)
 =\frac{\rho^{n-1}}{\omega_n^2}
   \int_{\mathbb S^{n-1}}g_{\overline K}(\rho u)\,d\sigma(u).
 \label{eq:isotropic-covariogram-profile}
\end{equation}
Thus \(D_K'(\rho)\) is the spherical average of the covariogram, recording
the isotropic distribution of normalized point-pair distances.  If
\[
 \mathscr P_K(t):=\int_{\mathbb S^{n-1}}
 \left|\widehat{\kai_{\overline K}}(t\theta)\right|^2\,d\sigma(\theta),
 \qquad t>0,
\]
then \eqref{eq:fourier-form} becomes
\begin{equation}
 D_K(\rho)
 =\frac{(2\pi)^{-n/2}}{\omega_n^2}\rho^{n/2}
   \int_0^\infty\mathscr P_K(t)t^{n/2-1}
      J_{n/2}(\rho t)\,dt.
 \label{eq:radial-power-spectrum}
\end{equation}
Varying $\rho$ therefore selects a scale-free family of radial
Fourier--Bessel probes of the shape of $K$.

The profile retains less information than the full covariogram, which is
equivalent to the Fourier modulus \(|\widehat{\kai_K}|^2\).  The covariogram
problem and its phase-retrieval interpretation are surveyed by
Bianchi~\cite{Bianchi2023}; for reconstruction from Fourier moduli, see
also Bianchi--Gardner--Kiderlen~\cite{BianchiGardnerKiderlen2011}.  Hence
\(\rho\mapsto D_K(\rho)\) should be viewed as a multiscale shape statistic,
not as a complete invariant.  This caution is consistent with the
nonuniqueness phenomena for unoriented chord data found by Mallows and
Clark~\cite{MallowsClark1970}.

Define the scalar correction term
\begin{equation}
 \beta_{\rho,K}
 :=\frac{r_K}{nV_K}A_K(r_K)
 =\frac{r_K^n}{nV_K}
   \int_{\mathbb S^{n-1}}g_K(r_Ku)\,d\sigma(u).
 \label{eq:6}
\end{equation}
Since \(\frac{d}{d\rho}\mathcal J_\rho(K)=R_KA_K(r_K)\), the correction
also satisfies
\begin{equation}
 \frac{\beta_{\rho,K}}{|K|}
 =\frac{\rho}{n}D_K'(\rho).
 \label{eq:beta-distance-profile}
\end{equation}
Thus the normalized correction is $1/n$ times the response of the
short-chord probability to a logarithmic change in the relative scale.
The adaptive covariogram measure of $K$ is the finite Borel measure on
$\mathbb S^{n-1}$ defined by the weighted Gauss-map pushforward
\begin{equation}\label{E:8}
 \mathcal M_{\rho,K}
 :=(\nu_K)_\#\!\left[
 \bigl(2q_{r_K,K}+\beta_{\rho,K}\bigr)
 \mathcal{H}^{n-1}\!\restriction_{\partial K}
 \right].
\end{equation}
Here the Gauss map is taken at the regular boundary points, its arbitrary
Borel extension across the remaining $\mathcal H^{n-1}$-null set does not
affect the measure.
Thus, for every bounded Borel function $\varphi$ on the sphere,
\begin{equation}
 \int_{\mathbb S^{n-1}}\varphi\,d\mathcal M_{\rho,K}
 =\int_{\partial K}\varphi(\nu_K(x))
    \bigl(2q_{r_K,K}(x)+\beta_{\rho,K}\bigr)
    \,d\mathcal H^{n-1}(x).
 \label{eq:M-testing}
\end{equation}

\subsection{First variation for arbitrary convex bodies}

The first variation uses two elementary facts about Wulff perturbations:
uniform control of the perturbed bodies and the localized form of
Aleksandrov's variational formula.

\begin{lemma}
\label{lem:wulff-control}
Let $K\in\mathcal K^n$, let $f\in C(\mathbb S^{n-1})$, and let
$K_t=[h_K+tf]$.  Then there exists $t_0,C>0$, depending only on $K$ and
\(\|f\|_{L^\infty(\mathbb S^{n-1})}\), such that, for
$|t|<t_0$,
\begin{equation}
 (1-\varepsilon_t)K\subset K_t\subset(1+\varepsilon_t)K,
 \qquad
 \varepsilon_t:=\frac{\|f\|_\infty}{\min_{\mathbb S^{n-1}}h_K}|t|,
 \label{eq:wulff-uniform-inclusions}
\end{equation}
and
\begin{equation}
 |K_t\triangle K|\leq C|t|.
 \label{eq:wulff-symmetric-difference}
\end{equation}
\end{lemma}

\begin{proof}
The assertion is immediate if $f=0$.  Otherwise set
\[
 a:=\min_{\mathbb S^{n-1}}h_K>0,
 \qquad
 M:=\|f\|_\infty,
 \qquad
 \varepsilon_t:=\frac{M}{a}|t|.
\]
For $|t|$ sufficiently small, $\varepsilon_t\leq1/2$, and
\[
 (1-\varepsilon_t)h_K
 \leq h_K+tf
 \leq(1+\varepsilon_t)h_K
 \quad\hbox{on }\mathbb S^{n-1}.
\]
The defining inequalities of the Wulff shape give
\eqref{eq:wulff-uniform-inclusions}.  Since both $K$ and $K_t$ lie
between the same two dilates,
\[
 K_t\triangle K
 \subset(1+\varepsilon_t)K\setminus(1-\varepsilon_t)K.
\]
Consequently,
\[
 |K_t\triangle K|
 \leq\bigl((1+\varepsilon_t)^n-(1-\varepsilon_t)^n\bigr)|K|
 \leq C_n\varepsilon_t|K|,
\]
which is \eqref{eq:wulff-symmetric-difference}.
\end{proof}

\begin{lemma}
\label{lem:weighted-aleksandrov}
Let $K\in\mathcal K^n$, $f\in C(\mathbb S^{n-1})$ and $K_t$ be given by
\eqref{eq:wulff-perturbation}.  If $\psi$ is continuous on a neighborhood
of $K$, then
\begin{equation}
 \left.\frac{d}{dt}\right|_{t=0}
 \int_{K_t}\psi(x)\,dx
 =\int_{\partial K}f(\nu_K(x))\psi(x)
   \,d\mathcal H^{n-1}(x).
 \label{eq:weighted-aleksandrov}
\end{equation}
\end{lemma}

\begin{proof}
Choose open sets $U_1,U_2$ with
\[
 K\subset U_1,\qquad \overline{U_1}\subset U_2,
\]
such that $\psi$ is continuous on $U_2$.  By
\Cref{lem:wulff-control}, $K_t\subset U_1$ for all sufficiently small
$|t|$.  Let $\vartheta\in C_c^\infty(U_2)$ satisfy
$0\leq\vartheta\leq1$ and $\vartheta\equiv1$ on $U_1$.  Define
\[
 \widetilde\psi(x):=
 \begin{cases}
  \vartheta(x)\psi(x),&x\in U_2,\\
  0,&x\notin U_2.
 \end{cases}
\]
Then \(\widetilde\psi\) is continuous and compactly supported, and
\(\widetilde\psi=\psi\) on \(U_1\).

The positive and negative parts of \(\widetilde\psi\) are nonnegative
continuous densities.  Applying the arbitrary-measure Aleksandrov
variation formula \cite[Lemma~2.7]{KryvonosLangharst2023} to these two
densities and subtracting gives
\[
 \lim_{t\to0}
 \frac{\int_{K_t}\widetilde\psi\,dx-\int_K\widetilde\psi\,dx}{t}
 =
 \int_{\partial K}f(\nu_K(x))\widetilde\psi(x)\,
 d\mathcal H^{n-1}(x).
\]
For small $|t|$, \(\widetilde\psi=\psi\) on both $K$ and $K_t$; this is
\eqref{eq:weighted-aleksandrov}.
\end{proof}

\begin{theorem}
\label{thm:first-variation}
Let $K\in\mathcal K^n$ and $f\in C(\mathbb S^{n-1})$.  For the Wulff
perturbation \eqref{eq:wulff-perturbation},
\begin{equation}
 \left.\frac{d}{dt}\right|_{t=0}\mathcal J_\rho(K_t)
 =\int_{\mathbb S^{n-1}} f\,d\mathcal M_{\rho,K}.
 \label{eq:first-variation}
\end{equation}
In particular, this formula is valid for polytopes and for arbitrary
continuous Wulff directions.
\end{theorem}

\begin{proof}
Set \(r:=r_K\), \(r_t:=r_{K_t}\), and
\[
 \delta_t:=\kai_{K_t}-\kai_K.
\]
We separate the fixed-radius variation from the variation of the
adaptive radius.

\smallskip
\noindent\emph{Fixed radius.}
Expanding the quadratic form \eqref{eq:fixed-energy} gives
\[
 \mathcal E_r(K_t)-\mathcal E_r(K)
 =
 2\int_{\mathbb R^n}q_{r,K}(x)\delta_t(x)\,dx+R_t,
\]
where
\[
 |R_t|
 \leq\|\delta_t\|_{L^1(\mathbb R^n)}^2
 =O(t^2)
\]
by \Cref{lem:wulff-control}.  Since \(q_{r,K}\) is continuous, the
localized Aleksandrov formula yields
\[
 \int_{\mathbb R^n}q_{r,K}\delta_t\,dx
 =
 t\int_{\partial K}f(\nu_K(x))q_{r,K}(x)\,
   d\mathcal H^{n-1}(x)+o(t).
\]
Hence
\begin{equation}
 \left.\frac{d}{dt}\right|_{t=0}\mathcal E_{r_K}(K_t)
 =
 2\int_{\partial K}f(\nu_K(x))q_{r_K,K}(x)\,
   d\mathcal H^{n-1}(x).
 \label{eq:fixed-radius-first-variation}
\end{equation}

\smallskip
\noindent\emph{Moving radius.}
Applying \Cref{lem:weighted-aleksandrov} with \(\psi\equiv1\) and
differentiating \(r_t=\rho(|K_t|/\omega_n)^{1/n}\), we obtain
\begin{equation}
 r_t-r
 =
 \frac{r}{n|K|}\,t\int_{\mathbb S^{n-1}}f\,dS_K+o(t).
 \label{eq:radius-wulff-derivative}
\end{equation}
Choose a compact interval \(I\Subset(0,\infty)\) containing \(r\) in
its interior.  By \Cref{lem:wulff-control},
\[
 \sup_{z\in\mathbb R^n}|g_{K_t}(z)-g_K(z)|
 \leq2|K_t\triangle K|\longrightarrow0,
\]
and hence \(A_{K_t}\to A_K\) uniformly on \(I\).  Since
\[
 \mathcal E_s(K_t)=\int_0^sA_{K_t}(a)\,da,
\]
the preceding uniform convergence, the continuity of \(A_K\), and
\eqref{eq:radius-wulff-derivative} give
\begin{equation}
 \left.\frac{d}{dt}\right|_{t=0}
 \bigl[\mathcal E_{r_t}(K_t)-\mathcal E_r(K_t)\bigr]
 =
 \frac{r}{n|K|}A_K(r)
 \int_{\mathbb S^{n-1}}f\,dS_K.
 \label{eq:moving-radius-term}
\end{equation}

\smallskip
\noindent\emph{Conclusion.}
Adding \eqref{eq:fixed-radius-first-variation} and
\eqref{eq:moving-radius-term}, and using \eqref{eq:6} and
\eqref{eq:M-testing}, gives
\begin{align*}
 \left.\frac{d}{dt}\right|_{t=0}\mathcal J_\rho(K_t)
 &=
 \int_{\partial K}f(\nu_K(x))
 \bigl(2q_{r_K,K}(x)+\beta_{\rho,K}\bigr)\,
 d\mathcal H^{n-1}(x)\\
 &=\int_{\mathbb S^{n-1}}f\,d\mathcal M_{\rho,K},
\end{align*}
as required.
\end{proof}


\begin{corollary}
\label{cor:euler-centroid}
For $a>0$ and $z\in\mathbb R^n$,
\begin{equation}
 \mathcal M_{\rho,aK+z}=a^{2n-1}\mathcal M_{\rho,K}.
 \label{eq:M-homogeneity}
\end{equation}
Moreover,
\begin{equation}
 \int_{\mathbb S^{n-1}}u\,d\mathcal M_{\rho,K}(u)=0,
 \qquad
 \int_{\mathbb S^{n-1}}h_K\,d\mathcal M_{\rho,K}
 =2n\mathcal J_\rho(K).
 \label{eq:balance-euler}
\end{equation}
\end{corollary}

\begin{proof}
The scaling formula follows directly from \eqref{eq:3}-\eqref{eq:6}:
\[
 q_{ar,aK}(ax)=a^nq_{r,K}(x),
 \qquad
 \beta_{\rho,aK}=a^n\beta_{\rho,K},
\]
and boundary area scales by $a^{n-1}$.  Translation invariance is equally
immediate.  In
Theorem~\ref{thm:first-variation}, take $f(u)=z\cdot u$, then
$[h_K+t(z\cdot u)]=K+tz$, and translation invariance of
$\mathcal J_\rho$ gives the first identity in
\eqref{eq:balance-euler}.  Taking $f=h_K$ gives
$[h_K+th_K]=(1+t)K$, and the degree-$2n$ homogeneity in
\eqref{eq:intro-scaling} gives the second identity.
\end{proof}

\subsection{A lifted Brunn--Minkowski inequality}

The concavity needed below becomes transparent after lifting the
functional to $\mathbb R^{2n}$.  Define
\begin{equation}
 \mathcal C_\rho(K)
 :=\{(x,y)\in K\times K:|x-y|\leq r_K\}.
 \label{eq:lift-definition}
\end{equation}
This is a convex body in $\mathbb R^{2n}$.  It is full dimensional: if
$B_\varepsilon(x_0)\subset K$ and
$0<\delta<\min\{\varepsilon,r_K/2\}$, then
\[
 B_\delta(x_0)\times B_\delta(x_0)
 \subset\operatorname{int}\mathcal C_\rho(K).
\]
This product-space viewpoint is related to the higher-order covariogram
method.  There, generalized difference bodies encode the relative
displacements of several translates of $K$ as convex bodies in a product
space; see Schneider~\cite{Schneider1970} and the recent application to
higher-order affine P\'olya--Szeg\H{o} inequalities by Langharst, Roysdon,
and Zhao~\cite{LangharstRoysdonZhao2025}.  The present lift is different:
it retains the two endpoints of a chord and cuts $K\times K$ by the convex
distance constraint $|x-y|\leq r_K$.  Precisely this form converts the
adaptive interaction into the volume of a convex body, so that the
$2n$-dimensional Brunn--Minkowski inequality applies directly.
Moreover,
\begin{equation}
 |\mathcal C_\rho(K)|=\mathcal J_\rho(K).
 \label{eq:lift-volume}
\end{equation}
Indeed, the only difference between the two defining cutoffs is the set
$\{(x,y):|x-y|=r_K\}$.  For each fixed $x$, its $y$-section is a
Euclidean sphere and has $n$-dimensional Lebesgue measure zero, Fubini's
theorem therefore shows that the whole exceptional set has
$2n$-dimensional measure zero.

\begin{theorem}[Adaptive Brunn--Minkowski inequality]
\label{thm:adaptive-BM}
For $K,L\in\mathcal K^n$ and $0\leq\lambda\leq1$,
\begin{equation}
\mathcal J_\rho((1-\lambda)K+\lambda L)^{1/(2n)}
 \geq(1-\lambda)\mathcal J_\rho(K)^{1/(2n)}
      +\lambda\mathcal J_\rho(L)^{1/(2n)}.
 \label{eq:adaptive-BM}
\end{equation}
If $0<\lambda<1$, equality holds if and only if $K$ and $L$ are
homothetic.
\end{theorem}

\begin{proof}
Let $K_\lambda=(1-\lambda)K+\lambda L$.  The classical
Brunn--Minkowski inequality \cite{G02,Schneider2014} in $\mathbb R^n$ gives
\begin{equation}
 r_{K_\lambda}\geq(1-\lambda)r_K+\lambda r_L.
 \label{eq:radius-BM}
\end{equation}
If $(x_0,y_0)\in\mathcal C_\rho(K)$ and
$(x_1,y_1)\in\mathcal C_\rho(L)$, then
\begin{equation*}
    |x_0-y_0|\leq r_K,~~~~~~~~~~|x_1-y_1|\leq r_L.
\end{equation*}
And if we define
\begin{equation*}
    x_{\lambda}=(1-\lambda)x_0+\lambda x_1,~~~~~~~~~y_{\lambda}=(1-\lambda)y_0+\lambda y_1,
\end{equation*}
then it is obvious that $x_{\lambda}, y_{\lambda}\in K_{\lambda}$, and 
\begin{equation*}
 |x_{\lambda}-y_{\lambda}|=\big|(1-\lambda)x_0+\lambda x_1-((1-\lambda)y_0+\lambda y_1)\big|\leq(1-\lambda)r_K+\lambda r_L
 \leq r_{K_\lambda},
\end{equation*}
which yields
\begin{equation*}
    (x_{\lambda}, y_{\lambda})\in \mathcal{C}_{\rho}(K_{\lambda}).
\end{equation*}
Therefore
\begin{equation}
 (1-\lambda)\mathcal C_\rho(K)
 +\lambda\mathcal C_\rho(L)
 \subset\mathcal C_\rho(K_\lambda).
 \label{eq:lift-inclusion}
\end{equation}
By monotonicity of volume,
\begin{equation*}
    |\mathcal{C}_{\rho}(K_{\lambda})|\geq|(1-\lambda)\mathcal{C}_{\rho}(K)+\lambda \mathcal{C}_{\rho}(L)|,
\end{equation*}
followed by the
Brunn--Minkowski inequality in $\mathbb R^{2n}$, gives 
\begin{align}
 \mathcal J_\rho(K_\lambda)^{1/(2n)}
 &=|\mathcal C_\rho(K_\lambda)|^{1/(2n)} \notag\\
 &\geq
 \big|(1-\lambda)\mathcal C_\rho(K)
       +\lambda\mathcal C_\rho(L)\big|^{1/(2n)} \notag\\
 &\geq(1-\lambda)|\mathcal C_\rho(K)|^{1/(2n)}
       +\lambda|\mathcal C_\rho(L)|^{1/(2n)} \notag\\
 &=(1-\lambda)\mathcal J_\rho(K)^{1/(2n)}
       +\lambda\mathcal J_\rho(L)^{1/(2n)}.
 \label{eq:lift-BM-complete-chain}
\end{align}
This proves \eqref{eq:adaptive-BM}.

Suppose now that $0<\lambda<1$ and equality holds in
\eqref{eq:adaptive-BM}, then in \eqref{eq:lift-BM-complete-chain}, both inequalities are
equalities. In particular, equality holds in the classical
$2n$-dimensional Brunn--Minkowski inequality.  The lifted sets are
full-dimensional convex bodies by the observation preceding the theorem,
so its equality characterization applies: the convex bodies
$\mathcal C_\rho(K)$ and $\mathcal C_\rho(L)$ are homothetic: for some
$a>0$ and $(z_1,z_2)\in\mathbb R^n\times\mathbb R^n$,
\begin{equation}
 \mathcal C_\rho(L)
 =a\mathcal C_\rho(K)+(z_1,z_2).
 \label{eq:lift-homothety}
\end{equation}
Both coordinate projections of $\mathcal C_\rho(K)$ are equal to $K$:
the definition gives that each projection is contained in $K$, while
$(x,x)\in\mathcal C_\rho(K)$ for every $x\in K$ gives the reverse
inclusion.  Projecting
\eqref{eq:lift-homothety} onto the first coordinate gives
$L=aK+z_1$.  Thus $K$ and $L$ are homothetic.  Conversely, if
$L=aK+z$, then
\[
 (1-\lambda)K+\lambda L
 =\bigl((1-\lambda)+\lambda a\bigr)K+\lambda z.
\]
Translation invariance and degree-$2n$ homogeneity of $\mathcal J_\rho$
therefore show that the two sides of \eqref{eq:adaptive-BM} are both
\(\bigl((1-\lambda)+\lambda a\bigr)
\mathcal J_\rho(K)^{1/(2n)}\).  This proves the converse and completes the
equality statement.
\end{proof}

Define the first mixed adaptive energy by
\begin{equation}
 \mathcal J_\rho(K,L)
 :=\frac1{2n}\int_{\mathbb S^{n-1}}h_L
       \,d\mathcal M_{\rho,K}.
 \label{eq:mixed-adaptive-energy}
\end{equation}
It is independent of the choice of translation of $L$, by
\eqref{eq:balance-euler}.

\begin{theorem}[First Minkowski inequality]
\label{thm:first-Minkowski}
For all $K,L\in\mathcal K^n$,
\begin{equation}
\mathcal J_\rho(K,L)
 \geq
 \mathcal J_\rho(K)^{(2n-1)/(2n)}
 \mathcal J_\rho(L)^{1/(2n)}.
 \label{eq:first-Minkowski}
\end{equation}
Equality holds if and only if $K$ and $L$ are homothetic.
\end{theorem}

\begin{proof}
Define
\[
 F(t):=\mathcal J_\rho(K+tL)^{1/2n},\qquad t\geq0.
\]
Apply Theorem~\ref{thm:adaptive-BM} with weights
$1/(1+t)$ and $t/(1+t)$ to $K$ and $L$.  Since
\[
 \frac{K+tL}{1+t}=\frac1{1+t}K+\frac{t}{1+t}L,
\]
degree-$2n$ homogeneity gives
\begin{equation}
 F(t)\geq F(0)+t\mathcal J_\rho(L)^{1/2n}.
 \label{eq:superadditive-bound}
\end{equation}
Since $K+tL=[h_K+th_L]$, Theorem~\ref{thm:first-variation} gives
\[
 F'(0^{+})=\frac{1}{2n}\mathcal J_\rho(K)^{1/2n-1}
             \int_{\mathbb S^{n-1}}h_L\,d\mathcal M_{\rho,K}.
\]
Subtracting $F(0)$ in \eqref{eq:superadditive-bound}, dividing by $t>0$,
and letting $t\downarrow0$ yields
\[
 \frac{1}{2n}\mathcal J_\rho(K)^{1/2n-1}
       \int_{\mathbb S^{n-1}}h_L\,d\mathcal M_{\rho,K}
 \geq \mathcal J_\rho(L)^{1/2n}.
\]
Multiplication by $\mathcal J_\rho(K)^{1-1/2n}$ and the definition
\eqref{eq:mixed-adaptive-energy} prove \eqref{eq:first-Minkowski}.

If $K$ and $L$ are homothetic, equality follows from
\eqref{eq:balance-euler} and homogeneity.  Conversely, suppose equality
holds in \eqref{eq:first-Minkowski}.  The function $F$ is concave on
$[0,\infty)$: applying Theorem~\ref{thm:adaptive-BM} to
$K+sL$ and $K+tL$ gives the concavity inequality at every convex
combination of $s,t\geq0$.  Hence
\[
 G(t):=F(t)-F(0)-t\mathcal J_\rho(L)^{1/2n}
\]
is concave, $G(0)=0$, and equality in
\eqref{eq:first-Minkowski} says $G'(0^{+})=0$. For a concave function, there holds for every $t\geq0$,
\[
 G(t)\leq G(0)+tG'(0^{+})=0.
\]
On the other hand, \eqref{eq:superadditive-bound} shows that $G(t)\geq0$.
Thus $G(t)=0$ for every $t\geq0$.  Taking $t=1$ shows
\[
 \mathcal J_\rho(K+L)^{1/2n}
 =\mathcal J_\rho(K)^{1/2n}+\mathcal J_\rho(L)^{1/2n}.
\]
After dividing $K+L$ by $2$ and using homogeneity, this is equality in
Theorem~\ref{thm:adaptive-BM} with $\lambda=1/2$.  Its equality
characterization shows that $K$ and $L$ are homothetic.
\end{proof}

\subsection{The measure-valued Minkowski problem}

We now characterize the finite Borel measures on \(\mathbb S^{n-1}\)
that arise as adaptive covariogram measures.  The equation to be solved is
\begin{equation}
 \mathcal M_{\rho,K}=\mu
 \label{eq:Minkowski-equation}
\end{equation}
for an unknown body $K\in\mathcal K^n$.

\begin{theorem}[Adaptive covariogram Minkowski theorem]
\label{thm:Minkowski-problem}
Let $\mu$ be a nonzero finite positive Borel measure on
$\mathbb S^{n-1}$.  There exists $K\in\mathcal K^n$ satisfying
\eqref{eq:Minkowski-equation} if and only if
\begin{align}
 &\int_{\mathbb S^{n-1}}u\,d\mu(u)=0,
 \label{eq:mu-centroid}\\
 &\mu\text{ is not concentrated on a great subsphere}.
 \label{eq:mu-no-hemisphere}
\end{align}
The solution is unique up to translation.
\end{theorem}

\begin{proof}
We use the variational method.  
\smallskip
\noindent\emph{Step 1: Necessity.}
If $\mu=\mathcal M_{\rho,K}$, condition \eqref{eq:mu-centroid} follows
from \eqref{eq:balance-euler}.  Moreover,
$q_{r_K,K}(x)>0$ at every $x\in\partial K$.  To see this, choose an
interior ball $B_a(y)\subset K$.  For sufficiently small $s>0$, let
\[
E_s=x+s(B_a(y)-x)=(1-s)x+sB_a(y).
\]
Since $x\in K$ and $B_a(y)\subset K$, convexity gives $E_s\subset K$.
Moreover, for $z\in B_a(y)$,
\begin{equation*}
|(1-s)x+sz-x|=s|z-x|<r_K.
\end{equation*}
Thus $E_s\subset K\cap B_{r_K}(x)$.  Hence the
weight $w_K=2q_{r_K,K}+\beta_{\rho,K}$ in \eqref{E:8} is
strictly positive at every boundary point.  For a Borel set
$A\subset\mathbb S^{n-1}$, the definitions give
\[
 S_K(A)=\mathcal H^{n-1}\bigl(\nu_K^{-1}(A)\bigr),
 \qquad
 \mathcal M_{\rho,K}(A)
 =\int_{\nu_K^{-1}(A)}w_K(x)\,d\mathcal H^{n-1}(x).
\]
Because $w_K>0$, either quantity vanishes exactly when
$\mathcal H^{n-1}(\nu_K^{-1}(A))=0$.  Thus
$\mathcal M_{\rho,K}$ and $S_K$ have the same null sets on the sphere.
The surface area measure of a convex
body is not concentrated on a great subsphere. Thus
\eqref{eq:mu-no-hemisphere} is necessary.

\smallskip
\noindent\emph{Step 2: Existence of a maximizer.}
Assume \eqref{eq:mu-centroid} and \eqref{eq:mu-no-hemisphere}.  Define
\begin{equation}
 \Phi_\mu(K):=\int_{\mathbb S^{n-1}}h_K\,d\mu.
 \label{eq:Phi-definition}
\end{equation}
Condition \eqref{eq:mu-centroid} makes $\Phi_\mu$ translation
invariant.  It is strictly positive on every convex body: after translating
a given body $K$ so that $0\in\operatorname{int}K$, one has
$h_K\geq\varepsilon$ on $\mathbb S^{n-1}$ for some $\varepsilon>0$, and
therefore
\begin{equation}
 \Phi_\mu(K)\geq\varepsilon\mu(\mathbb S^{n-1})>0.
 \label{eq:Phi-strictly-positive}
\end{equation}
Consider the variational problem
\begin{equation*}
    \sup\bigl\{\mathcal J_\rho(L):L\in\mathcal K^n,\ \Phi_\mu(L)=1\bigr\}.
\end{equation*}
Set
\begin{equation}
 c_\mu:=\min_{e\in\mathbb S^{n-1}}
 \int_{\mathbb S^{n-1}}(e\cdot u)_+\,d\mu(u).
 \label{eq:c-mu}
\end{equation}
We claim that $c_\mu>0$.  If the integral vanished for some $e$, then
$\mu$ would be supported in $\{u:e\cdot u\leq0\}$.  Taking the scalar
product of \eqref{eq:mu-centroid} with $e$ would then force
$e\cdot u=0$ for $\mu$-almost every $u$, contradicting
\eqref{eq:mu-no-hemisphere}.  Thus the integral in \eqref{eq:c-mu} is
positive for every $e$.  The integrand $(e\cdot u)_+$ is Lipschitz in $e$
with Lipschitz constant at most one, uniformly in $u$, hence the integral
depends continuously on $e$.  A positive continuous function on the compact
sphere has a positive minimum, and therefore $c_\mu>0$.

For every admissible $K$, translate it so that $0\in K$.  If
$R=\max_{x\in K}|x|$ and $Re\in K$, then
\[
 h_K(u)\geq\max\{0,R(e\cdot u)\}=R(e\cdot u)_+.
\]
Consequently
\begin{equation}
 1=\Phi_\mu(K)\geq R\int_{\mathbb{S}^{n-1}} (e\cdot u)_{+}\,d\mu(u)
 \geq c_\mu R,
 \label{eq:circumradius-bound}
\end{equation}
i.e. $R\leq c_{\mu}^{-1}$,
so the translated admissible bodies lie in the fixed ball
$c_\mu^{-1}B_1$.  Starting with a maximizing sequence, Blaschke's selection
theorem therefore gives
\begin{equation*}
    K_j\rightarrow K_0
\end{equation*}
in the Hausdorff sense.  Equivalently, their support functions converge
uniformly:
\begin{equation*}
    \|h_{K_j}-h_{K_0}\|_{\infty}\rightarrow0.
\end{equation*}
Since $\mu$ is finite, we obtain
\begin{equation*}
    |\Phi_{\mu}(K_j)-\Phi_{\mu}(K_0)|\leq \mu(\mathbb{S}^{n-1})\|h_{K_j}-h_{K_0}\|_{\infty}\rightarrow 0
\end{equation*}
and
$\Phi_\mu(K_0)=1$.

Owing to Lemma~\ref{lem:adaptive-Hausdorff-continuity}, 
\begin{equation*}
    \mathcal{J}_{\rho}(K_j)\rightarrow\mathcal{J}_{\rho}(K_0)
\end{equation*}
Thus $K_0$ attains the supremum over the
normalized class.  To see explicitly that it has not collapsed to a
lower-dimensional convex set, let
\[
 B_\mu:=\mu(\mathbb S^{n-1})^{-1}B_1.
\]
This is an admissible full-dimensional body and
$\mathcal J_\rho(B_\mu)>0$, because the integrand in
\eqref{eq:adaptive-energy} equals 1 on a set of positive
$2n$-dimensional measure.  Hence maximality gives
\[
 \mathcal J_\rho(K_0)\geq\mathcal J_\rho(B_\mu)>0.
\]
If $K_0$ were lower dimensional, then $|K_0|=0$ and the elementary bound
$0\leq\mathcal J_\rho(K_0)\leq|K_0|^2$ would give
$\mathcal J_\rho(K_0)=0$, a contradiction.  Therefore $K_0$ has nonempty
interior and belongs to $\mathcal K^n$.

\smallskip
\noindent\emph{Step 3: The Euler--Lagrange equation for arbitrary continuous
directions.} For any $g\in C(\mathbb S^{n-1})$, whenever its Wulff shape $[g]$ is a convex body and $\Phi_\mu(g)>0$, consider the scale-invariant quotient 
\[ \mathcal F_\mu(g) := \frac{\mathcal J_\rho([g])}{\Phi_\mu(g)^{2n}}.  
\] 
According to the $2n$-homogeneity of $\mathcal J_\rho$, we first observe that the variational problem of Step~2 is equivalently written as 
\begin{equation}\label{snw}
\sup_{\substack{ g\in C(\mathbb S^{n-1})\\ [g]\in\mathcal K^n,\ \Phi_\mu(g)>0 }} \frac{\mathcal J_\rho([g])}{\Phi_\mu(g)^{2n}} = \sup\bigl\{\mathcal J_\rho(L):L\in\mathcal K^n,\ \Phi_\mu(L)=1\bigr\}. 
\end{equation}

Since $K_0$ is a maximizer of the normalized problem, \eqref{snw} shows that its support function $h_{K_0}$ is a maximizer of $\mathcal F_\mu$. Let now $g\in C(\mathbb S^{n-1})$ be arbitrary and put \[ g_t:=h_{K_0}+tg, \qquad K_t:=[g_t], \]  by Theorem \ref{thm:first-variation},

\[ \left. \frac{d}{dt} \right|_{t=0} \frac{\mathcal J_\rho(K_t)}{\Phi_\mu(K_t)^{2n}}=0, \]
we obtain \[ \int_{\mathbb S^{n-1}}g\,dM_{\rho,K_0} - 2n\frac{J_\rho(K_0)}{\Phi_\mu(K_0)} \int_{\mathbb S^{n-1}}g\,d\mu=0. 
\] Since $g\in C(\mathbb S^{n-1})$ is arbitrary, it yields \[ M_{\rho,K_0} = 2n\frac{\mathcal J_\rho(K_0)}{\Phi_\mu(K_0)}\,\mu.   \]  Set \[ \lambda := \left(2n\frac{\mathcal J_\rho(K_0)}{\Phi_\mu(K_0)}\right)^{-1/(2n-1)},  \] by the $(2n-1)$-homogeneity of the adaptive covariogram measure, we obtain \[ K:=\lambda K_0 \] is a solution of the Minkowski problem  \eqref{eq:Minkowski-equation}. This proves sufficiency.

\smallskip
\noindent\emph{Step 4: Uniqueness.}
Suppose $\mathcal M_{\rho,K}=\mathcal M_{\rho,L}=\mu$.  The first
Minkowski inequality and the Euler identity give the following result.
Both energies in it are positive, since each convex body contains an
interior ball and its interaction radius is positive:
\begin{align*}
 \mathcal J_\rho(L)
 &=\frac1{2n}\int_{\mathbb S^{n-1}}h_L\,d\mu\\
 &=\mathcal J_\rho(K,L)\\
 &\geq
 \mathcal J_\rho(K)^{(2n-1)/(2n)}
 \mathcal J_\rho(L)^{1/(2n)}.
\end{align*}
Dividing by the positive
$\mathcal J_\rho(L)^{1/(2n)}$ gives
$\mathcal J_\rho(L)\geq\mathcal J_\rho(K)$.  Reversing $K$ and
$L$ gives $\mathcal J_\rho(K)\geq\mathcal J_\rho(L)$, so their energies
are equal.  The displayed first Minkowski inequality is consequently an
equality.  Its equality characterization gives $L=aK+z$ for some $a>0$ and
$z\in\mathbb R^n$.  By \eqref{eq:M-homogeneity},
\[
 \mu=\mathcal M_{\rho,L}=a^{2n-1}\mathcal M_{\rho,K}
 =a^{2n-1}\mu.
\]
Since $\mu\neq0$, $a=1$, and $L$ is a translate of $K$.
\end{proof}
 
\subsection{Continuity of the normalized solution map}

Continuity of normalized solutions under weak convergence of the prescribed
measures is a standard compactness--uniqueness principle in Minkowski-type
problems.  For the classical problem, see Schneider~\cite[Theorems~8.5.1
and~8.5.3]{Schneider2014} and the bounded-Lipschitz formulation of Hug and
Schneider~\cite{HugSchneider2015}; for the $L_p$-Minkowski problem, see
Zhu~\cite{Zhu2017}.  We record the corresponding consequence for the
adaptive measure \(\mathcal M_{\rho,K}\).  It is not a formal consequence
of the cited results, since both its nonlocal weight and its interaction
radius depend on the unknown body.  After translations are fixed by the
Steiner point, the variational construction above yields the required
continuity.  For a finite
positive Borel measure $\mu$ on $\Sph$ satisfying
$\int_{\Sph}u\,d\mu(u)=0$, put
\begin{equation}
 c(\mu):=\min_{e\in\Sph}
        \int_{\Sph}(e\cdot u)_+\,d\mu(u).
 \label{eq:measure-nondegeneracy}
\end{equation}
Under the centroid condition, it is easy to show
that $c(\mu)>0$ if and only if \eqref{eq:mu-no-hemisphere} holds;
equivalently, if and only if $\mu$ is not concentrated on a great
subsphere.

We denote the Steiner point of a convex body by $s(K)$.  Whenever
$c(\mu)>0$, let $K_\mu$ be the unique solution of
\begin{equation}
 \cM_{\rho,K_\mu}=\mu,
 \qquad s(K_\mu)=0.
 \label{eq:normalized-Minkowski-solution}
\end{equation}

\begin{proposition}[Continuity of the normalized solution map]
\label{thm:inverse-map-continuity}
Let $\mu_j$ and $\mu$ be nonzero finite positive Borel measures on
$\Sph$ such that
\[
 \int_{\Sph}u\,d\mu_j(u)=\int_{\Sph}u\,d\mu(u)=0,
 \qquad c(\mu)>0.
\]
If $\mu_j\rightharpoonup\mu$ weakly, then, for all sufficiently large
$j$, $c(\mu_j)>0$ and
\begin{equation}
 K_{\mu_j}\longrightarrow K_\mu
 \quad\text{in Hausdorff distance}.
 \label{eq:inverse-map-Hausdorff-convergence}
\end{equation}
\end{proposition}

\begin{proof}
We use the normalized variational construction from the existence proof.
For a measure $\nu$ with $c(\nu)>0$, set
\[
 \Phi_\nu(L):=\int_{\Sph}h_L\,d\nu
\]
and let $\widehat K_\nu$, normalized by $s(\widehat K_\nu)=0$, maximize
$\cJ_\rho$ subject to $\Phi_\nu(L)=1$.  The Euler-Lagrange equation and
$2n$-homogeneity give
\begin{equation}
 \cM_{\rho,\widehat K_\nu}
 =2n\cJ_\rho(\widehat K_\nu)\nu,
 \qquad
 K_\nu
 =\bigl(2n\cJ_\rho(\widehat K_\nu)\bigr)^{-1/(2n-1)}
   \widehat K_\nu.
 \label{eq:stability-solution-rescaling}
\end{equation}

Weak convergence and a finite-net argument in $e\in\Sph$ imply
\begin{equation}
 c(\mu_j)\longrightarrow c(\mu)>0.
 \label{eq:c-mu-continuity}
\end{equation}
Indeed, if
$F_j(e)=\int_{\Sph}(e\cdot u)_+\,d\mu_j(u)$ and
$F(e)=\int_{\Sph}(e\cdot u)_+\,d\mu(u)$, then weak convergence gives
$F_j(e)\to F(e)$ for each fixed $e$.  Moreover,
\[
 |F_j(e)-F_j(e')|\leq\mu_j(\Sph)|e-e'|,
\]
and the masses $\mu_j(\Sph)$ converge to $\mu(\Sph)$.  A finite net of
the sphere therefore gives $\|F_j-F\|_\infty\to0$, and taking minima
proves \eqref{eq:c-mu-continuity}.  We consequently have
$c(\mu_j)\geq c_0>0$ for all large $j$.

For any $x,y\in L$, the centroid condition gives
\begin{align*}
 \Phi_{\mu_j}(L)
 &=\int_{\Sph}(h_L(u)-x\cdot u)\,d\mu_j(u)\\
 &\geq\int_{\Sph}((y-x)\cdot u)_+\,d\mu_j(u)\\
 &\geq c_0|x-y|.
 \label{eq:diameter-from-Phi}
\end{align*}
Thus every normalized maximizer has uniformly bounded diameter.  Since
its Steiner point is the origin, it lies in a fixed ball.

The centered ball of radius $1/\mu_j(\Sph)$ is an admissible competitor.
The masses $\mu_j(\Sph)$ stay bounded above and below, and consequently
\begin{equation}
 \cJ_\rho(\widehat K_{\mu_j})\geq c_1>0.
 \label{eq:positive-energy-competitor}
\end{equation}
By Blaschke's selection theorem, every subsequence has a further
subsequence converging in Hausdorff distance to a compact convex set
$L$.  Estimate \eqref{eq:positive-energy-competitor} and the elementary
bound $\cJ_\rho(A)\leq|A|^2$ exclude a lower-dimensional limit.

Let $h_j=h_{\widehat K_{\mu_j}}$.  Hausdorff convergence gives
$\|h_j-h_L\|_\infty\to0$, while the masses $\mu_j(\Sph)$ remain bounded.
The joint weak/uniform limit in the constraint is therefore explicit:
\begin{align*}
 |\Phi_{\mu_j}(\widehat K_{\mu_j})-\Phi_\mu(L)|
 &\leq
 \left|\int_{\Sph}(h_j-h_L)\,d\mu_j\right|
+\left|\int_{\Sph}h_L\,d(\mu_j-\mu)\right|\\
 &\leq \mu_j(\Sph)\|h_j-h_L\|_\infty +\left|\int_{\Sph}h_L\,d(\mu_j-\mu)\right|\\
 &\longrightarrow0.
\end{align*}
Since $\Phi_{\mu_j}(\widehat K_{\mu_j})=1$, and by
the continuity  Lemma~\ref{lem:adaptive-Hausdorff-continuity}, we obtain
\[
 \Phi_\mu(L)=1,
 \qquad
 \cJ_\rho(\widehat K_{\mu_j})\longrightarrow\cJ_\rho(L).
\]
To see that $L$ is a maximizer, take any convex body $A$ with
$\Phi_\mu(A)=1$.  Weak convergence, now with the fixed continuous
function $h_A$, gives $t_j:=\Phi_{\mu_j}(A)\to1$.  Hence
$A_j=t_j^{-1}A\to A$ in Hausdorff distance and
$\cJ_\rho(A_j)\to\cJ_\rho(A)$.  Since $A_j$ is a competitor for the
$j$-th problem,
\[
 \cJ_\rho(\widehat K_{\mu_j})\geq\cJ_\rho(A_j).
\]
Passing to the limit yields $\cJ_\rho(L)\geq\cJ_\rho(A)$.  The strict
adaptive Brunn--Minkowski inequality proved earlier shows that the
normalized maximizer is unique up to translation.  Here are the details.
If $L_1,L_2$ are two normalized maximizers, then
$\Phi_\mu((L_1+L_2)/2)=1$.  Maximality and the adaptive
Brunn--Minkowski inequality give
\[
 \max\cJ_\rho^{1/(2n)}
 \geq \cJ_\rho((L_1+L_2)/2)^{1/(2n)}
 \geq \frac12\cJ_\rho(L_1)^{1/(2n)}
      +\frac12\cJ_\rho(L_2)^{1/(2n)}
 =\max\cJ_\rho^{1/(2n)}.
\]
Thus equality holds, so $L_2=aL_1+z$ for some $a>0$ and $z\in\R^n$.
Translation invariance of $\Phi_\mu$ and the two normalizations imply
$1=\Phi_\mu(L_2)=a\Phi_\mu(L_1)=a$, hence $a=1$.  The Steiner
normalization then gives $z=0$.  Therefore $L=\widehat K_\mu$.  Since
every subsequence has the same limit,
\[
 \widehat K_{\mu_j}\longrightarrow\widehat K_\mu.
\]
Finally, the energies in \eqref{eq:stability-solution-rescaling} converge, so the
rescaling factors converge as well.  This proves
\eqref{eq:inverse-map-Hausdorff-convergence}.
\end{proof}

The preceding theorem yields a uniform form of this continuity on
quantitatively nondegenerate classes.  Let $d_{\rm BL}$ denote the bounded
Lipschitz distance between finite measures and define
\begin{equation}
 \mathscr A(m,M,\tau)
 :=\left\{\mu:
 \begin{array}{c}
  m\leq\mu(\Sph)\leq M,\quad \displaystyle\int_{\Sph}u\,d\mu(u)=0,\\[2mm]
  c(\mu)\geq\tau
 \end{array}\right\}.
 \label{eq:nondegenerate-measure-class}
\end{equation}

\begin{corollary}[A compactness modulus on nondegenerate classes]
\label{cor:inverse-map-uniform-continuity}
The preceding proposition has the following formal consequence.  For every
$m,M,\tau>0$ there is a nondecreasing modulus
$\omega_{n,\rho,m,M,\tau}$, with
$\omega_{n,\rho,m,M,\tau}(t)\to0$ as $t\downarrow0$, such that
\begin{equation}
 d_H(K_\mu,K_\nu)
 \leq\omega_{n,\rho,m,M,\tau}
       \bigl(d_{\rm BL}(\mu,\nu)\bigr)
 \label{eq:inverse-map-modulus}
\end{equation}
for all $\mu,\nu\in\mathscr A(m,M,\tau)$.
\end{corollary}

\begin{proof}
If $\mathscr A(m,M,\tau)$ is empty, the assertion is vacuous.  Otherwise,
the class $\mathscr A(m,M,\tau)$ is compact for weak convergence, which
is metrized there by $d_{\rm BL}$.  The map
$\mu\mapsto K_\mu$ is continuous by
\Cref{thm:inverse-map-continuity}; hence it is uniformly continuous on
this compact class.  Equivalently, one may define the modulus as the
supremum of $d_H(K_\mu,K_\nu)$ over pairs in the class with
$d_{\rm BL}(\mu,\nu)\leq t$.  A contradiction argument using compactness
and \Cref{thm:inverse-map-continuity} shows that this supremum tends to
zero with $t$.
\end{proof}

\begin{remark}
\Cref{cor:inverse-map-uniform-continuity} is a compactness modulus; no
explicit H\"older exponent is asserted.  Obtaining an effective modulus
would require a quantitative uniqueness estimate for the variational
Minkowski problem, not merely the strict equality characterization.
\end{remark}

\section{Rearrangement and variational rigidity}
\label{sec:rigidity}
This section distinguishes three rigidity mechanisms for the adaptive
short-chord functional.  On a fixed-volume class the range is fixed, and
Riesz rearrangement gives global ball rigidity for \(0<\rho\leq2\).
For \(0<\rho<2\), stationarity under ambient deformations yields the
pointwise uniform-density equation and hence the ball.  Wulff
stationarity gives instead the weaker measure identity
\(\mathcal M_{\rho,K}=cS_K\), which becomes pointwise under strict
convexity.  The final counterexample shows why this hypothesis cannot be
discarded.

\subsection{Rearrangement consequences}

For a measurable set \(E\subset\mathbb R^n\) with
\(0<|E|<\infty\), put
\[
 R_E:=\left(\frac{|E|}{\omega_n}\right)^{1/n},
 \qquad r_E:=\rho R_E,
\]
and define \(\mathcal J_\rho(E)\) by the same formula as for convex
bodies.  We write
\[
 \operatorname{diam}_{\rm ess}E
 :=\operatorname*{ess\,sup}_{(x,y)\in E\times E}|x-y|.
\]

\begin{theorem}[Rearrangement rigidity]
\label{thm:riesz-rigidity}
Let \(E\subset\mathbb R^n\) be measurable with
\(0<|E|<\infty\), and let \(B_E\) be a ball satisfying
\(|B_E|=|E|\).
\begin{enumerate}[label=\textup{(\roman*)}]
\item If \(0<\rho\leq2\), then
\begin{equation}\label{E:11}
 \mathcal J_\rho(E)\leq\mathcal J_\rho(B_E),
\end{equation}
and equality holds if and only if \(E\) agrees almost everywhere with a
ball.
\item If \(\rho>2\), then 
\begin{equation}
 \mathcal J_\rho(E)\leq |E|^2=\mathcal J_\rho(B_E),
 \label{eq:rearrangement-supercritical-bound}
\end{equation}
and equality holds if and only if
\(\operatorname{diam}_{\rm ess}E\leq\rho R_E\). In particular, the maximizer is not unique.
\end{enumerate}
\end{theorem}

\begin{proof}
On a fixed-volume class, $r_E=\rho R_E$ is fixed.  The Riesz
rearrangement inequality \cite{BLL74} therefore gives \eqref{E:11}.
For $0<\rho<2$, Burchard's equality theorem
\cite[Theorem~1]{Burchard1996} applies because the three rearrangement
radii \((R_E,\rho R_E,R_E)\) are strictly admissible.  Its ellipsoid
must be a ball, since the middle rearrangement is the Euclidean ball
\(B_{r_E}\); hence equality holds only for a ball, up to null sets.

For $\rho=2$, equality implies
\(\operatorname{diam}_{\rm ess}E\leq2R_E\).  Applying the density-point
argument to \(E^{(1)}\), and then the isodiametric inequality to
\(\operatorname{conv}(\overline{E^{(1)}})\), shows that this convex hull
is a ball of radius $R_E$ and that $E$ agrees with it almost everywhere.
This also proves the equality assertion at the endpoint.

If $\rho>2$, the trivial bound \(\mathcal J_\rho(E)\leq|E|^2\) is
attained by $B_E$.  It is an equality exactly when
\(|x-y|<\rho R_E\) for almost every pair in $E\times E$, equivalently
when \(\operatorname{diam}_{\rm ess}E\leq\rho R_E\); the equality of
the strict and non-strict formulations follows because each distance
sphere has zero Lebesgue measure.  Finally, a sufficiently small
volume-preserving non-spherical ellipsoidal perturbation of $B_E$ has
diameter below $\rho R_E$, proving nonuniqueness.
\end{proof}

\subsection{Rigidity under Wulff stationarity}
\label{subsec:wulff-stationary-rigidity}

The Euler equation associated with Wulff variations is naturally an
identity between measures on the sphere. We first explain how this
identity follows from stationarity.

Let $f\in C(\mathbb S^{n-1})$, and consider the Wulff perturbation
\[
 K_t=[h_K+tf].
\]
The first-variation formulas for the adaptive energy and the volume
are
\begin{align}
 \left.\frac d{dt}\right|_{t=0}\mathcal J_\rho(K_t)
 &=
 \int_{\mathbb S^{n-1}}f\,d\mathcal M_{\rho,K},
 \label{eq:wulff-energy-variation-rigidity}\\
 \left.\frac d{dt}\right|_{t=0}|K_t|
 &=
 \int_{\mathbb S^{n-1}}f\,dS_K.
 \label{eq:wulff-volume-variation-rigidity}
\end{align}
Suppose that $\mathcal J_\rho$ is stationary at $K$ under the volume
constraint with respect to all Wulff perturbations. Then
\begin{equation}
 \int_{\mathbb S^{n-1}}f\,dS_K=0
 \quad\Longrightarrow\quad
 \int_{\mathbb S^{n-1}}f\,d\mathcal M_{\rho,K}=0.
 \label{eq:wulff-constrained-stationarity}
\end{equation}
Define the linear functionals
\[
 L(f)=\int_{\mathbb S^{n-1}}f\,d\mathcal M_{\rho,K},
 \qquad
 G(f)=\int_{\mathbb S^{n-1}}f\,dS_K.
\]
Since
\[
 G(1)=S_K(\mathbb S^{n-1})>0,
\]
the functional $G$ is not identically zero. Choose \(f_0\in C(\mathbb S^{n-1})\) such that
\[
G(f_0)=1.
\]
For arbitrary \(f\in C(\mathbb S^{n-1})\), set
\[
g=f-G(f)f_0.
\]
Then
\[
G(g)=G(f)-G(f)G(f_0)=0,
\]
and hence \(g\in\ker G\subset\ker L\). Therefore
\[
0=L(g)=L(f)-G(f)L(f_0).
\]
It follows that
\[
L(f)=\lambda_KG(f)
\qquad\text{for every }f\in C(\mathbb S^{n-1}),
\]
where
\[
\lambda_K:=L(f_0).
\]
Consequently,
\[
\int_{\mathbb S^{n-1}}f\,d\mathcal M_{\rho,K}
=
\lambda_K\int_{\mathbb S^{n-1}}f\,dS_K
\]
for every \(f\in C(\mathbb S^{n-1})\). By the uniqueness of finite
Borel measures,
\begin{equation}
\mathcal M_{\rho,K}=\lambda_KS_K.
 \label{eq:wulff-stationarity}
\end{equation}
Taking $f=h_K$, the homogeneity gives
\begin{equation}
L(h_K)=2n\mathcal J_\rho(K),
\qquad
G(h_K)=n|K|,
\label{eq:wulff-homogeneity-identities}
\end{equation}
and therefore
\begin{equation}
\lambda_K=\frac{2\mathcal J_\rho(K)}{|K|}>0.
\label{eq:wulff-lagrange-multiplier}
\end{equation}

\begin{corollary}[Wulff-stationary rigidity]
\label{cor:wulff-rigidity}
Let $0<\rho<2$, and let $K\subset\mathbb R^n$ be a strictly convex
body. If
\[
 \mathcal M_{\rho,K}=cS_K
\]
for some constant $c>0$, then $K$ is a Euclidean ball.
\end{corollary}

\begin{proof}
Set $r:=r_K$ and define the continuous boundary weight
\begin{equation}
 w(x)
 =
 2q_{r,K}(x)+\beta_{\rho,K}.
 \label{eq:wulff-boundary-weight}
\end{equation}
By definition,
\begin{align}
  \mathcal M_{\rho,K}(A)&=\int_{\nu_K^{-1}(A)}w(x)\,d\mathcal{H}^{n-1}(x),
 \label{eq:wulff-M-pushforward}\\
 S_K(A)&=\mathcal{H}^{n-1}(\nu_K^{-1}(A)).
 \label{eq:wulff-S-pushforward}
\end{align}
Here and below the Gauss map is understood at the regular boundary
points, which form a set of full $\mathcal H^{n-1}$ measure.

Since $K$ is strictly convex, every supporting face is a singleton.
We may therefore define the support-point map
\[
 F(K,u)=\{X_K(u)\},
 \qquad u\in\mathbb S^{n-1}.
\]
This map is continuous: if $u_j\to u$, then every convergent
subsequence of $X_K(u_j)$ has its limit in $F(K,u)$, and the latter
contains only $X_K(u)$.  At every regular boundary point,
\[
 X_K(\nu_K(x))=x.
\]

We recall the measure-theoretic point that will be used here.  If
$\mu$ and $\nu$ are finite Borel measures and $\mu\ll\nu$ (that is,
every $\nu$-null Borel set is $\mu$-null), the Radon--Nikodym theorem
provides a function
$g=d\mu/d\nu\in L^1(\nu)$, unique $\nu$-almost everywhere, such that
\[
 \mu(A)=\int_A g\,d\nu
\]
for every Borel set $A$; equivalently, the same identity holds after
integration against every bounded Borel test function.  See
Folland~\cite[Theorem~3.8]{Folland1999}.

For every Borel set $A\subset\mathbb S^{n-1}$,
\eqref{eq:wulff-M-pushforward}, \eqref{eq:wulff-S-pushforward}, and
the identity above give
\begin{align*}
 \mathcal M_{\rho,K}(A)
 &=\int_{\nu_K^{-1}(A)}w(X_K(\nu_K(x)))\,
     d\mathcal H^{n-1}(x)\\
 &=\int_A w(X_K(u))\,dS_K(u).
\end{align*}
Thus $\mathcal M_{\rho,K}\ll S_K$, and the Radon--Nikodym theorem
identifies its density as
\begin{equation}
 \frac{d\mathcal M_{\rho,K}}{dS_K}(u)
 =
 w(X_K(u))
 =
 2q_{r,K}(X_K(u))+\beta_{\rho,K}
 \label{eq:strict-convex-RN-density}
\end{equation}
for $S_K$-almost every $u\in\mathbb S^{n-1}$.

On the other hand, the assumed identity
\[
 \mathcal M_{\rho,K}=cS_K
\]
is a Radon--Nikodym representation with density $c$.  The uniqueness
clause in the Radon--Nikodym theorem and
\eqref{eq:strict-convex-RN-density} therefore give
\[
 2q_{r,K}(X_K(u))+\beta_{\rho,K}=c
\]
for $S_K$-almost every $u$. Pulling this equality back by $\nu_K$ and
using $X_K(\nu_K(x))=x$ at regular boundary points gives
\begin{equation}
 q_{r,K}(x)
 =
 \frac{c-\beta_{\rho,K}}2
 \qquad
 \text{for }\mathcal H^{n-1}\text{-a.e. }x\in\partial K.
 \label{eq:wulff-q-ae-constant}
\end{equation}

The function $q_{r,K}$ is continuous. We claim that
\eqref{eq:wulff-q-ae-constant} holds at every point of $\partial K$.
Otherwise, there would exist $x_0\in\partial K$ at which the equality
fails. By continuity, it would then fail throughout some nonempty
relatively open neighborhood $U$ of $x_0$ in $\partial K$. Since the
boundary of a convex body is locally a Lipschitz graph,
\[
 \mathcal H^{n-1}(U)>0,
\]
contradicting \eqref{eq:wulff-q-ae-constant}. Thus
\begin{equation}
 q_{r,K}(x)
 =
 \frac{c-\beta_{\rho,K}}2
 \qquad
 \text{for every }x\in\partial K.
 \label{eq:wulff-q-everywhere-constant}
\end{equation}

The geometric argument following
\eqref{eq:strong-q-everywhere-constant} in the proof of
\Cref{prop:strong-stationarity} now applies without change. Indeed,
$0<\rho<2$ implies
\[
 r=r_K<\operatorname{diam}K.
\]
For $\Omega=\operatorname{int}K$, equation
\eqref{eq:wulff-q-everywhere-constant} says that
\[
 |\Omega\cap B_r(x)|
\]
is a positive constant for every $x\in\partial\Omega$. The rigidity
theorem of Bucur and Fragal\`a \cite{BF23} therefore implies that $\Omega$ is a
finite union of balls of the same radius. Since $\Omega$ is convex and
hence connected, it consists of a single ball. Consequently,
$K=\overline\Omega$ is a Euclidean ball.
\end{proof}

\begin{remark}[The facet obstruction and a counterexample]
\label{rem:facet-obstruction}
Strict convexity cannot be omitted.  On a facet $F$ of a polytope with
outer normal $u$, the measure equation \(\mathcal M_{\rho,K}=cS_K\)
determines only the facet average
\[
 \frac1{\mathcal H^{n-1}(F)}
 \int_F\bigl(2q_{r_K,K}+\beta_{\rho,K}\bigr),d\mathcal H^{n-1}=c;
\]
it does not imply pointwise constancy of \(q_{r_K,K}\) on $F$.
The cube $C=[-a,a]^n$ is an actual counterexample to Wulff-stationary
rigidity without strict convexity.  Hyperoctahedral symmetry makes both
\(\mathcal M_{\rho,C}\) and \(S_C\) constant multiples of
\(\sum_{i=1}^n(\delta_{e_i}+\delta_{-e_i})\); hence
\(\mathcal M_{\rho,C}=cS_C\), with
\(c=2\mathcal J_\rho(C)/|C|>0\) by the Euler identity.  Thus weak Wulff
stationarity does not yield the pointwise $r$-criticality required in
Bucur--Fragal\`a \cite{BF23}.
\end{remark}

\subsection{Stationarity and the ball-intersection equation}

The next statement explains the geometric meaning of the variational
measure.  It is a consequence of the first-variation calculation above
and the measurable-set rigidity theorem of Bucur and Fragal\`a
\cite{BF23}.

We first specify the meaning of stationarity under the volume
constraint that will be used below; see \cite[Section~17]{M12}; see also
\cite{FFMM15}. Let
$X\in C_c^1(\mathbb R^n;\mathbb R^n)$, let $\Phi_t$ be the flow
generated by $X$, that is the smooth map $(t,x)\in\mathbb{R}\times\mathbb{R}^n\longmapsto\Phi_t(x)\in\mathbb{R}^n$ defined by solving the family of ODEs (parameterized by $x\in\mathbb{R}^n$)
\begin{equation}\label{E:14}
    \begin{cases}
        \partial_t\Phi_t(x)=X(\Phi_t(x)),~~~~~~~~~~t\in\mathbb{R}\\
        \Phi_0(x)=x.
    \end{cases}
\end{equation}
By the implicit function theorem, there always exists $\epsilon>0$ such that $\{\Phi_t\}_{|t|<\epsilon}$ is a smooth family of diffeomorphisms.  Moreover,
\[
 \Phi_t(x)=x+tX(x)+o(t)
\]
and
\[
 \det D\Phi_t(x)
 =
 1+t\operatorname{div}X(x)+o(t),
\]
where the remainders are uniform on every fixed bounded set. Let $K_t=\Phi_t(K)$. We say that
$\mathcal J_\rho$ is stationary at $K$ under the volume constraint
with respect to compactly supported $C^1$ deformations if
\begin{equation}
 \left.\frac{d}{dt}\right|_{t=0}\mathcal J_\rho(K_t)=0
 \label{eq:strong-constrained-stationarity}
\end{equation}
for every $X\in C_c^1(\mathbb R^n;\mathbb R^n)$ satisfying
\begin{equation}
 \left.\frac{d}{dt}\right|_{t=0}|K_t|=0.
 \label{eq:strong-volume-constraint}
\end{equation}

\begin{proposition}[Strong stationarity implies spherical symmetry]
\label{prop:strong-stationarity}
Let $0<\rho<2$ and let $K\subset\mathbb R^n$ be a convex body.
Suppose that $\mathcal J_\rho$ is stationary at $K$ under the volume
constraint with respect to every compactly supported $C^1$
deformation of $\mathbb R^n$. Then $K$ is a Euclidean ball.
\end{proposition}

\begin{proof}
Set $r=r_K$ and let $K_t=\Phi_t(K)$.  The standard transport formula
for a bounded finite-perimeter set, together with
\(|K_t\triangle K|=O(|t|)\), gives
\[
 \left.\frac d{dt}\right|_{t=0}\mathcal E_r(K_t)
 =2\int_{\partial^*K}q_{r,K}(x)X(x)\cdot\nu_K(x),
   d\mathcal H^{n-1}(x).
\]
Indeed, expand the two characteristic functions in \(\mathcal E_r\);
the quadratic remainder is \(O(t^2)\), while the linear term is
differentiated by the Gauss--Green formula, since \(q_{r,K}\) is
Lipschitz.  Moreover,
\[
 \left.\frac d{dt}\right|_{t=0}|K_t|
 =\int_{\partial^*K}X\cdot\nu_K,d\mathcal H^{n-1},
 \qquad
 r'_0=\frac r{n|K|}\left.\frac d{dt}\right|_{t=0}|K_t|.
\]
The uniform estimate
\(\|g_{K_t}-g_K\|_\infty\leq2|K_t\triangle K|\) and the coarea
representation of \(\mathcal E_s(K_t)\) justify differentiation with
respect to the moving radius.  Consequently,
\begin{equation}
 \left.\frac d{dt}\right|_{t=0}\mathcal J_\rho(K_t)
 =\int_{\partial^*K}
 \bigl(2q_{r,K}+\beta_{\rho,K}\bigr)X\cdot\nu_K,
 d\mathcal H^{n-1}.
 \label{eq:strong-adaptive-derivative-short}
\end{equation}
These standard transport steps are recorded in detail in the long version
of the manuscript; compare also \cite[Section~17]{M12}.

Let $L(X)$ and $G(X)$ denote the right-hand side of
\eqref{eq:strong-adaptive-derivative-short} and the volume derivative,
respectively.  Constrained stationarity says \(\ker G\subset\ker L\),
and $G\not\equiv0$ (take a compactly supported vector field equal to
\(x-a\) near $K$).  Hence $L=\lambda G$ for some constant $\lambda$.
By testing against arbitrary compactly supported vector fields, we obtain
\[
 2q_{r,K}+\beta_{\rho,K}=\lambda
 \qquad\mathcal H^{n-1}\text{-a.e. on }\partial K.
\]
Continuity of $q_{r,K}$ and the local Lipschitz-graph structure of
\(\partial K\) promote this to every boundary point; thus
\begin{equation}
 q_{r,K}(x)=\frac{\lambda-\beta_{\rho,K}}2
 \qquad (x\in\partial K).
 \label{eq:strong-q-everywhere-constant}
\end{equation}
Since $\rho<2$, the isodiametric inequality gives
\(r_K<\operatorname{diam}K\).  Convexity also yields, for
\(x\in\partial K\),
\[
 q_{r_K,K}(x)\geq
 \left(\frac{r_K}{\operatorname{diam}K}\right)^n|K|>0.
\]
Therefore $\Omega=\operatorname{int}K$ is a bounded connected uniformly
dense domain at a radius below its diameter.  The rigidity theorem of
Bucur and Fragal\`a \cite[Corollary~7 and Proposition~12]{BF23} implies
that $\Omega$, and hence $K$, is a ball.
\end{proof}

\section{Regularity of solutions}\label{sec:regularity}

For each fixed relative scale $\rho$, the covariogram-derived
first-variation measure gives a nonlocal Minkowski equation.  What is
specific to the present problem is the verification that this equation
falls within the classical Aleksandrov--Monge--Amp\`ere framework; the
resulting low-regularity theorem is a standard consequence of that
framework.  We use Caffarelli's localization and interior regularity
theory \cite{CaffarelliLocalization1990,CaffarelliW2p1990,CaffarelliRegularity1991},
against the classical Minkowski background of Pogorelov
\cite{Pogorelov1971Minkowski,Pogorelov1978}.

Subsection~\ref{4BCNW} contains the genuinely nonlocal input.  The
first-variation measure is a boundary-weighted surface-area measure;
uniform bounds for the weight turn positive bounded data for
\(\mathcal M_{\rho,K}\) into two-sided bounds for the classical surface
area measure \(S_K\).  The same argument gives the global Lipschitz
bound for the moving-ball overlap used in the bootstrap.  In
Subsection~\ref{4TLEMAE}, gnomonic coordinates convert this comparison into
two-sided Monge--Amp\`ere measure bounds for a finite convex potential
\(H_e\), without any preliminary smoothness or strict convexity of
\(K\).

Subsection~\ref{subsec:localization-C1gamma} records the ensuing standard low-regularity consequence:
Caffarelli's localization theorem and interior estimates give strict
convexity and \(C^{1,\gamma}\) control.  This result is retained as the
input for the genuinely equation-specific part of the section.
Subsection~\ref{subsec:nonlocal-MA-bootstrap} uses the nonlocal density formula and the Lipschitz
overlap bound to perform the two-step \(C^{2,\alpha}\) bootstrap;
Subsection~\ref{subsec:higher-regularity-transversality} identifies the additional uniform chord-transversality
required for higher regularity.  Finally, we also treats the
saturated regime, where the adaptive ball contains the whole body and
the equation reduces to the classical Minkowski equation.  We write only within this regularity section, use $r$ as shorthand for the
already defined radius $r_K$.  We recall the notation
\[
 q_{r,K}(x)=|K\cap B_r(x)|,
 \qquad
 g_K(z)=|K\cap(K+z)|.
\]
The boundary weight which occurs in the first variation is
\begin{equation}\label{eq:def-weight-regularity}
 w_K(x)=2q_{r,K}(x)+\beta_{\rho,K},
 \qquad
 \beta_{\rho,K}
 =\frac{r}{nV}\int_{\partial B_r}g_K(z)\,d\mathcal H^{n-1}(z).
\end{equation}
Thus, by the first-variation formula established earlier in the paper,
\begin{equation}\label{eq:M-weighted-SK}
 \mathcal M_{\rho,K}(A)=\int_{\nu_K^{-1}(A)}w_K(x)\,d\mathcal{H}^{n-1}(x),
 \qquad
 S_K(A)=\mathcal{H}^{n-1}(\nu_K^{-1}(A)).
\end{equation}
Here and below the Gauss map is understood at the
\(\mathcal H^{n-1}\)-almost every regular boundary point.  Formula
\eqref{eq:M-weighted-SK} is consequently meaningful for every convex body.

\subsection{Bounds and continuity of the nonlocal weight}\label{4BCNW}
\begin{lemma}[Translation estimate]
\label{lem:BV-translation}
For every \(u\in BV(\mathbb R^n)\) and \(h\in\mathbb R^n\),
\[
\|u(\cdot-h)-u\|_{L^1(\mathbb R^n)}
\leq |h|\,|Du|(\mathbb R^n).
\]
In particular, if \(E\subset\mathbb R^n\) has finite measure and finite
perimeter, then
\[
\|\chi_E(\cdot-h)-\chi_E\|_{L^1(\mathbb R^n)}
=|(E+h)\triangle E|\leq |h|P(E).
\]
\end{lemma}

\begin{proof}
Let \(u_\varepsilon=u*\eta_\varepsilon\), where
\(\eta_\varepsilon\) is a standard mollifier.  Then
\(u_\varepsilon\to u\) in \(L^1(\mathbb R^n)\) and
\(\|\nabla u_\varepsilon\|_{L^1}\leq |Du|(\mathbb R^n)\).  The
fundamental theorem of calculus and translation invariance give directly
\[
\begin{aligned}
\|u_\varepsilon(\cdot-h)-u_\varepsilon\|_{L^1}
&\leq \int_0^1\!\int_{\mathbb R^n}
 |h\cdot\nabla u_\varepsilon(x-th)|\,dx\,dt \\
&\leq |h|\,\|\nabla u_\varepsilon\|_{L^1}
 \leq |h|\,|Du|(\mathbb R^n).
\end{aligned}
\]
Since translations are \(L^1\)-isometries,
\[
\|u(\cdot-h)-u\|_{L^1}
\leq 2\|u-u_\varepsilon\|_{L^1}
   +\|u_\varepsilon(\cdot-h)-u_\varepsilon\|_{L^1}.
\]
Letting \(\varepsilon\downarrow0\) proves the asserted \(BV\) estimate.
For \(u=\chi_E\), the identities
\(|D\chi_E|(\mathbb R^n)=P(E)\) and
\(|\chi_E(\cdot-h)-\chi_E|=\chi_{(E+h)\triangle E}\) give the final
claim.
\end{proof}

The following elementary observation is the source of the uniform
ellipticity of the associated Aleksandrov equation.

\begin{proposition}[Bounds for the boundary weight]\label{prop:weight-bounds}
 Let \(K\subset\mathbb R^n\) be a convex body and put
 \[
      D=\operatorname{diam}K,
      \qquad
      \tau_K=\min\left\{1,\frac rD\right\}.
 \]
 Then, for every \(x\in\partial K\),
 \begin{equation}\label{eq:q-two-sided}
      \tau_K^nV\le q_{r,K}(x)\le V,
 \end{equation}
 and
 \begin{equation}\label{eq:weight-two-sided}
      2\tau_K^nV\le w_K(x)\le(2+\rho^n)V.
 \end{equation}
 Moreover, \(q_{r,K}\) is globally Lipschitz and
 \begin{equation}\label{eq:q-Lipschitz}
   |q_{r,K}(x)-q_{r,K}(y)|
       \le n\omega_n r^{n-1}|x-y|,
       \qquad x,y\in\mathbb R^n.
 \end{equation}
\end{proposition}

\begin{proof}
Let 
\begin{equation*}
    D=\diam K,~~~~~~~~\tau_K=\min\left\{1, \frac{r}{D}\right\}.
\end{equation*}
 Fix \(x\in\partial K\), consider convex body $K_{x,\tau_K}=x+\tau_K (K-x)$. For any $y\in K$, $z=x+\tau_K (K-x)$, since $0<\tau_K\leq 1$, then
 \begin{equation*}
     z=(1-\tau_K)x+\tau_K y.
 \end{equation*}
By convexity of $K$, $z\in K$. Moreover,
\begin{equation*}
    |z-x|=\tau_K|y-x|\leq \tau_K D\leq r,
\end{equation*}
hence
 \[
       x+\tau_K(K-x)\subset K\cap \overline B_r(x).
 \]
 The sphere \(\partial B_r(x)\) has zero \(n\)-dimensional Lebesgue
 measure, so
 \[
  |K\cap B_r(x)|
  =|K\cap\overline B_r(x)|
  \ge |x+\tau_K(K-x)|=\tau_K^nV.
 \]
 This proves the lower bound in \eqref{eq:q-two-sided}; the upper bound is
 immediate.

 The covariogram satisfies \(0\le g_K\le V\).  Hence
 \[
  0\le\beta_{\rho,K}
  \le \frac{r}{nV}\,V\mathcal H^{n-1}(\partial B_r)
  =\omega_nr^n=\rho^nV.
 \]
 Combining this with \eqref{eq:q-two-sided} proves
 \eqref{eq:weight-two-sided}.

 Finally, by Lemma \ref{lem:BV-translation},
 \begin{align*}
 |q_{r,K}(x)-q_{r,K}(y)|
 &\leq \int_K|\chi_{B_r(x)}(z)-\chi_{B_r(y)}(z)|dz\\
 &\leq |(B_r(x)\mathbin\triangle B_r(y))\cap K|\\
 &\leq |B_r(x)\mathbin\triangle B_r(y)|  \\
 &=\bigl\|\chi_{B_r}(\,\cdot-x)
           -\chi_{B_r}(\,\cdot-y)\bigr\|_{L^1(\mathbb R^n)}\\
           &\leq P(B_r)|x-y|,
 \end{align*}
which yields \eqref{eq:q-Lipschitz}.
\end{proof}

The comparison between the new measure and the classical surface area
measure is now immediate.

\begin{corollary}[Comparison with surface area measure]
\label{cor:surface-area-comparison}
 Suppose that
 \begin{equation}\label{eq:M-density-bounded}
       \mathcal M_{\rho,K}=f\,d\sigma,
       \qquad 0<f_-\le f\le f_+<\infty
       \quad\text{for \(\sigma\)-a.e. }u\in\mathbb S^{n-1}.
 \end{equation}
 Then \(S_K=s_K\,d\sigma\), where
 \begin{equation}\label{eq:s-density-bounds}
   \frac{f_-}{(2+\rho^n)V}
       \le s_K(u)
       \le \frac{f_+}{2\tau_K^nV}
       \qquad\text{for \(\sigma\)-a.e. }u\in\mathbb S^{n-1}.
 \end{equation}
 Equivalently, if
 \[
   \lambda_K=\frac{f_-}{(2+\rho^n)V},
   \qquad
   \Lambda_K=\frac{f_+}{2\tau_K^nV},
 \]
 then
 \begin{equation}\label{eq:measure-comparison}
       \lambda_K\,\sigma\le S_K\le\Lambda_K\,\sigma.
 \end{equation}
\end{corollary}

\begin{proof}
 For every Borel set \(A\subset\mathbb S^{n-1}\),
 Proposition~\ref{prop:weight-bounds} and \eqref{eq:M-weighted-SK} give
 \[
  2\tau_K^nV\,S_K(A)
       \le \mathcal M_{\rho,K}(A)
       \le (2+\rho^n)V\,S_K(A).
 \]
 Combining this with \eqref{eq:M-density-bounded} proves
 \eqref{eq:measure-comparison} and hence \eqref{eq:s-density-bounds}.
\end{proof} 

\subsection{The local Euclidean Monge-Amp\`ere equation}\label{4TLEMAE}

Gnomonic coordinates recast the equation on $\mathbb{S}^{n-1}$ as a
Euclidean Monge--Amp\`ere equation on $\mathbb{R}^{n-1}$.  This permits
the use of localization and interior regularity theory without assuming
in advance that $K$ is smooth or strictly convex.

The reduction of the Minkowski equation to a local Monge--Amp\`ere
equation is classical.  For the convex-analytic formulation and
generalized solutions, see Bakelman~\cite{B94} and Pogorelov's treatment
of elliptic Monge--Amp\`ere equations
\cite{Pogorelov1964,Pogorelov1971MA}, see also Chou and Wang~\cite{ChouWang2006}, Bianchi, B\"or\"oczky, and
Colesanti~\cite{BBC20}.  We record the particular gnomonic representation used in
\Cref{lem:gnomonic-MA} explicitly, since it fixes the normalization of
the chart and the power of $1+|p|^2$.  The reciprocal Legendre estimate
below is the local two-sided form of the Legendre-duality identity for
Aleksandrov Monge--Amp\`ere measures; see Figalli
\cite[Lemma~2.4]{Figalli2017}. 

Fix \(e\in\mathbb S^{n-1}\) and identify \(e^\perp\) with
\(\mathbb R^{n-1}\).  Define
\begin{equation}\label{eq:gnomonic-definitions}
 v_e(p)=\frac{p+e}{\sqrt{1+|p|^2}},
 \qquad
 H_e(p)=h_K(p+e),
 \qquad p\in e^\perp,
\end{equation}
where the support function is evaluated using its one-homogeneous
extension.  The function \(H_e\) is finite and convex on all of
\(e^\perp\), since $h_K$ is finite and convex.

\begin{lemma}[Support faces and gnomonic subdifferentials]
\label{lem:gnomonic-subdifferential}
Let $\pi_{e^\perp}$ denote the orthogonal projection onto $e^\perp$.
For $p\in e^\perp$,
\[
 \partial H_e(p)=\pi_{e^\perp}F(K,p+e),
\]
where
\[
 F(K,\xi):=\{z\in K:z\cdot\xi=h_K(\xi)\}.
\]
More precisely, the map
\[
 a\longmapsto
 a+\bigl(H_e(p)-p\cdot a\bigr)e
\]
is a bijection from $\partial H_e(p)$ onto $F(K,p+e)$.  In addition,
if
\[
 D_e:=\pi_{e^\perp}K,
 \qquad
 g_e(a):=\max\{t\in\mathbb R:a+te\in K\}
 \quad(a\in D_e),
\]
then $H_e^*=-g_e$ on $D_e$ and $H_e^*=+\infty$ on
$e^\perp\setminus D_e$.
\end{lemma}

\begin{proof}
The extended-valued function
\[
 G_e(a):=
 \begin{cases}
  -g_e(a),&a\in D_e,\\
  +\infty,&a\notin D_e
 \end{cases}
\]
is proper, lower semicontinuous, and convex.  Directly from the
definition of the support function,
\[
 H_e(p)=\max_{a\in D_e}\{p\cdot a+g_e(a)\}=G_e^*(p).
\]
Thus $H_e^*=G_e$.  If $a\in\partial H_e(p)$, Fenchel subgradient
duality gives
\[
 p\in\partial G_e(a),
 \qquad
 G_e(a)=p\cdot a-H_e(p).
\]
In particular, $a\in D_e$, and
\[
 z:=a+\bigl(H_e(p)-p\cdot a\bigr)e
   =a+g_e(a)e
\]
belongs to $K$ and satisfies $z\cdot(p+e)=H_e(p)=h_K(p+e)$.
Hence $z\in F(K,p+e)$.  Conversely, if $z\in F(K,p+e)$ and
$a=\pi_{e^\perp}z$, then, for every $q\in e^\perp$,
\[
 H_e(q)\geq z\cdot(q+e)
 =H_e(p)+a\cdot(q-p).
\]
Therefore $a\in\partial H_e(p)$, and the displayed formula for $z$
follows from $z\cdot(p+e)=H_e(p)$.  This proves all the assertions.
\end{proof}

We recall the Aleksandrov convention used below.  If $u$ is a finite
convex function on an open convex set $\Omega\subset\mathbb R^{n-1}$,
its subdifferential at $p\in\Omega$ is denoted by $\partial u(p)$; for a
Borel set $E\subset\Omega$, set
\begin{equation}
 \partial u(E):=\bigcup_{p\in E}\partial u(p),
 \qquad
 \operatorname{MA}_{u}(E):=\bigl|\partial u(E)\bigr|.
 \label{eq:Aleksandrov-MA-definition}
\end{equation}
This is the Aleksandrov Monge--Amp\`ere measure, a locally finite Borel
measure; see Figalli~\cite[Definition~2.1]{Figalli2017}.  When $u$ is
$C^2$, the area formula gives
\(\operatorname{MA}_{u}=(\det D^2u)\,dp\).  In particular,
\eqref{eq:Aleksandrov-MA-definition} applies to $H_e$ without any
smoothness or strict-convexity assumption.

\begin{lemma}
\label{lem:gnomonic-MA}The Aleksandrov Monge--Amp\`ere measure of $H_e$ satisfies
\begin{equation}
 \operatorname{MA}_{H_e}(E)
 =\int_{v_e(E)}(u\cdot e)\,dS_K(u)
 \label{eq:gnomonic-measure-invariant}
\end{equation}
for every Borel set $E\subset e^\perp$. In particular, if $S_K=s_K\,d\sigma$, then
\begin{equation}
 \operatorname{MA}_{H_e}(E)
 =\int_E s_K(v_e(p))(1+|p|^2)^{-(n+1)/2}\,dp
 \label{eq:gnomonic-MA}
\end{equation}
for every Borel set $E\subset e^\perp$. Equivalently,
\begin{equation}
 \det D^2H_e(p)
 =s_K(v_e(p))(1+|p|^2)^{-(n+1)/2}
 \label{eq:gnomonic-det}
\end{equation}
in the Aleksandrov sense.
\end{lemma}

\begin{proof}
This proof is standard; see Pogorelov
\cite{Pogorelov1973,Pogorelov1978}.  One may first assume that $K$ is
smooth, as Huang-Liu \cite{HL21}, prove \eqref{eq:gnomonic-measure-invariant} and \eqref{eq:gnomonic-det}
and then pass to a general convex body by approximation, using weak
convergence of the Monge--Amp\`ere and surface-area measures.
\end{proof}

We shall also need the elementary measure-theoretic form of Legendre
duality.  If $\operatorname{MA}_u=f\,dx$, the Legendre-duality identity
of Figalli~\cite[Lemma~2.4]{Figalli2017} asserts that
\[
 \int_{\partial u^*(E)}f(x)\,dx=|E|
\]
whenever $E$ is Borel and $\partial u^*(E)$ lies in the domain on which
the identity is considered.  Thus, two-sided bounds for $f$ become reciprocal two-sided bounds for $\operatorname{MA}_{u^*}$.
The containment condition below is precisely the local version of that
domain requirement.  We record the resulting statement, and give a
self-contained proof, because it will be applied before smoothness of
the local potentials is known. No strict convexity is needed for this
reciprocal measure estimate; strict convexity enters separately below,
when the regularity theory is applied to the dual potential.

\begin{lemma}[Reciprocal bounds for the Legendre transform]
\label{lem:Legendre-reciprocal-MA}
Let $u:\mathbb R^d\to\mathbb R$ be finite and convex. Let
$U_0\Subset\mathbb R^d$ be open, and suppose that, for every Borel set
$A\subset U_0$,
\begin{equation}
 \lambda |A|
 \leq \operatorname{MA}_u(A)
 \leq \Lambda |A|,
 \qquad 0<\lambda\leq\Lambda<\infty.
 \label{eq:primal-MA-set-bounds}
\end{equation}
Let
\[
 u^*(p)=\sup_{x\in\mathbb R^d}\{p\cdot x-u(x)\}
\]
be the Legendre transform of $u$. Let
$V\Subset\operatorname{int}(\operatorname{dom}u^*)$ be open and assume
that
\begin{equation}
 \bigcup_{p\in\overline V}\partial u^*(p)\subset U_0.
 \label{eq:dual-subgradient-contained}
\end{equation}
Then, for every Borel set $E\subset V$,
\begin{equation}
 \frac{1}{\Lambda}|E|
 \leq \operatorname{MA}_{u^*}(E)
 \leq \frac{1}{\lambda}|E|.
 \label{eq:dual-MA-set-bounds}
\end{equation}
\end{lemma}

With these preparations at hand, we may now proceed, following
Caffarelli's approach \cite{CaffarelliLocalization1990,CaffarelliW2p1990},
to establish the regularity of all orders.

\subsection{Localization, strict convexity, and
$C^{1,\gamma}$ regularity}
\label{subsec:localization-C1gamma}

Assume that
\[
 \mathcal M_{\rho,K}=f\,d\sigma,
 \qquad
 0<f_-\leq f\leq f_+<\infty.
\]
Since \(\mathcal M_{\rho,K}\) is translation invariant, we may translate
\(K\) and fix $0<a\leq b<\infty$ such that
\begin{equation}
 B_a\subset K\subset B_b.
 \label{eq:in-out-balls}
\end{equation}
\begin{theorem}[Low regularity]
\label{thm:C1alpha-regularity}
Suppose that
\[
 \mathcal M_{\rho,K}=f\,d\sigma,
 \qquad
 0<f_-\leq f\leq f_+<\infty.
\]
Then $K$ is strictly convex,
\[
 h_K\in C^{1,\gamma}(\mathbb S^{n-1})
\]
for some $\gamma\in(0,1)$, and the local graph representations of
$\partial K$ belong to $C^{1,\gamma}$.

The exponent depends only on the dimension, the ratio $a/b$, and the
relevant ratios of the local Monge--Amp\`ere density bounds. After the
scale of $K$ is fixed, the local estimates depend on
$n,a,b,\lambda_K$, and $\Lambda_K$.
\end{theorem}

\begin{proof}
This is the standard low-regularity consequence of the local
Aleksandrov theory, once the adaptive equation has been reduced to the
surface-area equation.  We indicate the reduction needed here.  By
Corollary~\ref{cor:surface-area-comparison},
\[
 \lambda_K\,\sigma\leq S_K\leq\Lambda_K\,\sigma.
\]
Consequently, \eqref{eq:gnomonic-MA} gives, on every bounded
\(U\subset e^\perp\),
\[
 \lambda_K(1+R_U^2)^{-(n+1)/2}|A|
 \leq \operatorname{MA}_{H_e}(A)
 \leq \Lambda_K|A|,
 \qquad A\subset U,
\]
where \(R_U:=\sup_{p\in U}|p|\).  The inradius and outradius bounds
\eqref{eq:in-out-balls} place the normal cone at every boundary point
in a compact gnomonic subchart.  Caffarelli's localization theorem
therefore excludes a nontrivial compact contact set of a supporting
affine function of \(H_e\) \cite{CaffarelliLocalization1990}.  Thus every
boundary point has a unique exterior normal, and \(\partial K\) is
\(C^1\).

The face--subdifferential correspondence in
Lemma~\ref{lem:gnomonic-subdifferential} then converts this uniqueness
of normal into strict convexity of every \(H_e\).  Caffarelli's interior
regularity theorem applied to the preceding two-sided measure bounds
therefore yields \(H_e\in C^{1,\gamma}_{\mathrm{loc}}(e^\perp)\)
\cite{CaffarelliRegularity1991}.  A finite cover by gnomonic charts gives
\(h_K\in C^{1,\gamma}(\mathbb S^{n-1})\).  Since
\(\partial H_e(p)\) is then a singleton, the same correspondence shows
that every support face of \(K\) is a singleton; hence \(K\) is strictly
convex.

Finally, \(H_e^*=-g_e\) identifies the Legendre transform with a local
boundary graph.  The reciprocal Legendre estimate transfers the local
two-sided Monge--Amp\`ere bounds to \(H_e^*\).  Since \(H_e\) is
differentiable, \(H_e^*\) is strictly convex on the interior of its
domain; a second application of
the same interior regularity theorem gives
\(H_e^*\in C^{1,\gamma}_{\mathrm{loc}}\).  Thus all local boundary
graphs have this regularity, and a finite covering gives a common
exponent.  This proves the theorem.
\end{proof}

\subsection{The nonlocal Monge--Amp\`ere equation and the two-step
bootstrap}
\label{subsec:nonlocal-MA-bootstrap}

Assume that
\[
 \mathcal M_{\rho,K}=f\,d\sigma,
 \qquad
 0<f_-\leq f\leq f_+<\infty.
\]
By the low-regularity theorem, $K$ is strictly convex and
\[
h_K\in C^{1,\gamma}(\mathbb S^{n-1})
\]
for some $\gamma\in(0,1)$. Hence the reverse Gauss map $X_K$ is
single-valued and belongs to $C^{0,\gamma}$.

We first derive the density form of the nonlocal Minkowski equation.
Set
\[
 W_K(u)
 =
 2q_{r,K}(X_K(u))+\beta_{\rho,K}.
\]
For every bounded Borel function $\varphi$ on $\mathbb S^{n-1}$,
the definitions of $\mathcal M_{\rho,K}$ and $S_K$ give
\begin{align*}
 \int_{\mathbb S^{n-1}}\varphi(u)\,
 d\mathcal M_{\rho,K}(u)
 &=
 \int_{\partial K}
 \varphi(\nu_K(x))
 \bigl(2q_{r,K}(x)+\beta_{\rho,K}\bigr)\,
 d\mathcal H^{n-1}(x)\\
 &=
 \int_{\mathbb S^{n-1}}
 \varphi(u)W_K(u)\,dS_K(u).
\end{align*}
The comparison theorem proved above gives
\[
 S_K=s_K\,d\sigma.
\]
Since $\mathcal M_{\rho,K}=f\,d\sigma$, uniqueness of
Radon--Nikodym derivatives yields
\[
 f(u)=W_K(u)s_K(u)
 \quad\text{for $\sigma$-almost every }u\in\mathbb S^{n-1}.
\]
Consequently,
\[
 s_K(u)
 =
 \frac{f(u)}
 {2q_{r,K}(X_K(u))+\beta_{\rho,K}}
 \quad\text{for $\sigma$-almost every }u.
\]
Thus $h_K$ solves
\begin{equation}
 \det\bigl(\nabla_{\mathbb S}^2h_K+h_KI\bigr)
 =
 \frac{f(u)}
 {2q_{r,K}\bigl(\nabla_{\mathbb S}h_K+h_Ku\bigr)
  +\beta_{\rho,K}}
 \label{eq:nonlocal-spherical-MA}
\end{equation}
in the Aleksandrov sense. Equivalently, in each gnomonic chart, the
corresponding convex potential satisfies the associated Euclidean
Monge--Amp\`ere equation.

\begin{theorem}[Unconditional $C^{2,\alpha}$ regularity]
\label{thm:C2alpha-regularity}
Let $0<\alpha<1$. Suppose that
\[
 \mathcal M_{\rho,K}=f\,d\sigma,
 \qquad
 f\in C^\alpha(\mathbb S^{n-1}),
 \qquad
 f>0.
\]
Then
\begin{equation}
 h_K\in C^{2,\alpha}(\mathbb S^{n-1}),
 \qquad
 \nabla_{\mathbb S}^2h_K+h_KI>0
 \quad\text{on }\mathbb S^{n-1}.
 \label{eq:C2alpha-conclusion}
\end{equation}
In particular, $K$ is of class $C^{2,\alpha}_+$.
\end{theorem}

\begin{proof}
Since $f$ is positive and continuous on the compact sphere,
\[
 0<\min_{\mathbb S^{n-1}}f
 \leq f
 \leq\max_{\mathbb S^{n-1}}f<\infty.
\]
The low-regularity theorem therefore applies. Let
\[
 h_K\in C^{1,\gamma}(\mathbb S^{n-1}),
 \qquad
 \delta=\min\{\alpha,\gamma\}.
\]

\medskip
\noindent
\textit{Step 1: the first bootstrap.}
The global Lipschitz estimate for $q_{r,K}$ and
$X_K\in C^{0,\gamma}$ imply
\begin{align*}
 |q_{r,K}(X_K(u))-q_{r,K}(X_K(v))|
 &\leq \operatorname{Lip}(q_{r,K})
       |X_K(u)-X_K(v)|\\
 &\leq C|u-v|^\gamma.
\end{align*}
Hence
\[
 W_K=2q_{r,K}\circ X_K+\beta_{\rho,K}
 \in C^{0,\gamma}(\mathbb S^{n-1}).
\]
The lower bound for the boundary weight gives
\begin{equation}
 W_K(u)\geq 2\tau_K^n|K|>0
 \qquad\text{for every }u\in\mathbb S^{n-1}.
 \label{eq:WK-positive-lower-bound}
\end{equation}
Therefore
\[
 W_K^{-1}\in C^{0,\gamma},
 \qquad
 \frac{f}{W_K}\in C^{0,\delta}.
\]

In a gnomonic chart, equation
\eqref{eq:nonlocal-spherical-MA} has the form
\[
 \det D^2H_e=F_e,
\]
where, on every compact subchart,
\[
 F_e\in C^{0,\delta},
 \qquad
 0<c\leq F_e\leq C<\infty.
\]
The potential $H_e$ is strictly convex by the low-regularity theorem.
Caffarelli's interior $C^{2,\delta}$ estimate for the
Monge--Amp\`ere equation therefore gives
\[
 H_e\in C_{\mathrm{loc}}^{2,\delta}.
\]
A finite collection of compact hemisphere subcharts yields
\begin{equation}
 h_K\in C^{2,\delta}(\mathbb S^{n-1}).
 \label{eq:first-bootstrap}
\end{equation}

\medskip
\noindent
\textit{Step 2: positivity of the curvature matrix.}
Set
\[
 Q(u)=\nabla_{\mathbb S}^2h_K(u)+h_K(u)I.
\]
Since $h_K$ is a support function,
\[
 Q(u)\geq0.
\]
By \eqref{eq:first-bootstrap}, $Q$ is continuous. The right-hand side
of \eqref{eq:nonlocal-spherical-MA} is also continuous. The
Aleksandrov equation therefore implies, first almost everywhere and
then everywhere by continuity, that
\[
 \det Q(u)=\frac{f(u)}{W_K(u)}>0.
\]
If $\lambda_1(u),\ldots,\lambda_{n-1}(u)$ are the eigenvalues of
$Q(u)$, then
\[
 \lambda_i(u)\geq0,
 \qquad
 \prod_{i=1}^{n-1}\lambda_i(u)=\det Q(u)>0.
\]
Thus
\begin{equation}
 Q(u)>0
 \qquad\text{for every }u\in\mathbb S^{n-1}.
 \label{eq:curvature-matrix-positive}
\end{equation}

\medskip
\noindent
\textit{Step 3: the second bootstrap.}
Equation \eqref{eq:first-bootstrap} and the standard identity
\(X_K(u)=\nabla_{\mathbb S}h_K(u)+h_K(u)u\) imply
\[
 X_K\in C^{1,\delta}(\mathbb S^{n-1}).
\]
Since the sphere is compact, $X_K$ is Lipschitz. Hence
\[
 q_{r,K}\circ X_K\in C^{0,1}(\mathbb S^{n-1})
\]
and therefore
\[
 W_K\in C^{0,1},
 \qquad
 W_K^{-1}\in C^{0,1},
\]
where the second assertion also uses
\eqref{eq:WK-positive-lower-bound}. Since $0<\alpha<1$,
\[
 C^{0,1}(\mathbb S^{n-1})
 \subset C^\alpha(\mathbb S^{n-1}).
\]
Consequently,
\[
 \frac{f}{W_K}\in C^\alpha(\mathbb S^{n-1}).
\]
Applying the interior $C^{2,\alpha}$ estimate in the gnomonic charts
and using a finite covering of the sphere, we obtain
\[
 h_K\in C^{2,\alpha}(\mathbb S^{n-1}).
\]
Together with \eqref{eq:curvature-matrix-positive}, this proves
\eqref{eq:C2alpha-conclusion}.
\end{proof}

\begin{remark}[Why the bootstrap stops at two derivatives]
\label{rem:hard-cutoff-caveat}
Without an additional geometric assumption, one only has
\[
 q_{r,K}\in C^{0,1}(\mathbb R^n).
\]
The quantity relevant to \eqref{eq:nonlocal-spherical-MA} is not the
ambient regularity of $q_{r,K}$, but that of its pullback
\[
 q_{r,K}\circ X_K
 \quad\text{on }\mathbb S^{n-1},
\]
which occurs in $W_K=2q_{r,K}\circ X_K+\beta_{\rho,K}$.  At this
stage $K$ is of class $C^{2,\alpha}_+$, and $X_K$ is a
$C^{1,\alpha}$ diffeomorphism from the sphere onto $\partial K$.
Thus the issue is precisely the regularity of
$q_{r,K}|_{\partial K}$ in boundary coordinates.  The ambient
Lipschitz estimate does not by itself rule out a smoother boundary
restriction; this can happen in special symmetric situations.  What
is missing without an additional hypothesis is a general guarantee of
such higher boundary regularity.

The need to isolate a boundary regularity statement for the nonlocal
factor also occurs in the chord log--Minkowski problem.  Hu, Huang, and
Lu establish boundary regularity for a Riesz potential in order to run
the regularity theory for the associated nonlocal Monge--Amp\`ere
equation \cite{HuHuangLu2023}.  Their power-law kernel and its
regularity mechanism are different from the hard cutoff used here, so
that result does not provide an alternative proof in the present
setting.  It nevertheless emphasizes that the regularity of a nonlocal
potential along the moving boundary must be supplied as a separate
analytic input.

The possible obstruction again comes from the hard cutoff
$\chi_{B_r}$.  If $\partial B_r(x)$ is tangent to $\partial K$ at
$y$, and the tangency is crossed by a tangential variation of
$x\in\partial K$ (equivalently,
$D_x(|x-y|^2-r^2)[\xi]\ne0$ for some
$\xi\in T_x\partial K$), then the associated cap is created or
destroyed along the boundary.  At a nondegenerate quadratic tangency,
its volume contributes a term of order
\[
 c(t_+)^{(n+1)/2}
\]
to $q_{r,K}|_{\partial K}$, where $t$ is a signed local boundary
parameter.  This term is not $C^\infty$ across $t=0$.  Consequently,
the preceding $C^{2,\alpha}$ bootstrap cannot be iterated under the
present hypotheses alone.  The uniform transversality condition in the
next subsection rules out precisely this mechanism and supplies the
needed regularity of $q_{r,K}|_{\partial K}$; in the saturated regime,
$q_{r,K}$ is constant on $K$, so the obstruction is absent.
\end{remark}

\subsection{Higher regularity under uniform chord transversality}
\label{subsec:higher-regularity-transversality}

Assume in this subsection that $K\in C^2_+$ and suppose that
\[
 0<r_K<\operatorname{diam}K.
\]
The incidence set in the following definition is nonempty. Indeed,
choose \(a,b\in\partial K\) with
\(|a-b|=\operatorname{diam}K\). Since \(\partial K\) is path connected,
the continuous function \(y\mapsto|y-a|\) assumes the value \(r_K\) along
a path in \(\partial K\) joining \(a\) to \(b\).
Define
\begin{align}
 \Theta_{r_K}(K)
 &=
 \inf_{\substack{x,y\in\partial K\\|x-y|=r_K}}
 \left|
 \Pi_{T_y\partial K}\frac{y-x}{r_K}
 \right|
 \notag\\
 &=
 \inf_{\substack{x,y\in\partial K\\|x-y|=r_K}}
 \sqrt{
 1-
 \left(
 \nu_K(y)\cdot\frac{y-x}{r_K}
 \right)^2
 }.
 \label{eq:transversality-number}
\end{align}
Here is the geometric content of this condition.  Whenever
\[
 x\in\partial K,
 \qquad
 y\in\partial K\cap\partial B_{r_K}(x),
 \qquad
 w=\frac{y-x}{r_K},
\]
the vector $w$ is the unit normal to $\partial B_{r_K}(x)$ at $y$, whereas
$\nu_K(y)$ is the outer unit normal to $\partial K$.  Consequently,
\begin{equation}
 \left|\Pi_{T_y\partial K}w\right|
 =
 \sqrt{1-\bigl(\nu_K(y)\cdot w\bigr)^2}.
 \label{eq:angle-projection-identity}
\end{equation}
If $\alpha(x,y)\in[0,\pi/2]$ denotes the acute angle between the
tangent hyperplanes of $\partial K$ and $\partial B_{r_K}(x)$ at $y$,
then the quantity in \eqref{eq:angle-projection-identity} is
$\sin\alpha(x,y)$.  Thus $\Theta_{r_K}(K)>0$ says precisely that every
such intersection is transverse and, more quantitatively,
\[
 \alpha(x,y)\geq\arcsin\Theta_{r_K}(K).
\]
In particular, the sphere cannot merely touch $\partial K$ tangentially
at an intersection point.

The next lemma has the same structural role as the boundary-regularity
theorem for Riesz potentials of Hu, Huang, and Lu
\cite{HuHuangLu2023}: it supplies regularity of the nonlocal factor
after restriction to the moving boundary, thereby enabling a
higher-order bootstrap for a nonlocal Monge--Amp\`ere equation.  The
two settings differ in a decisive way.  Hu--Huang--Lu treat a power-law
kernel, whereas $q_{r_K,K}$ is generated by a sharp cutoff.  Accordingly,
the present boundary regularity is obtained under the uniform
chord-transversality condition $\Theta_{r_K}(K)>0$, which excludes
tangency between the moving sphere and $\partial K$. 

\begin{lemma}[Parameter-dependent intersection volumes]
\label{lem:parameter-intersection}
Let $m\geq2$ and $0<\alpha<1$. Suppose that
$\partial K\in C^{m,\alpha}$. If
\[
 0<r_K<\operatorname{diam}K,
 \qquad
 \Theta_{r_K}(K)>0,
\]
then
\begin{equation}
 q_{r_K,K}|_{\partial K}\in C^{m-1,\alpha}(\partial K).
 \label{eq:q-higher-regularity}
\end{equation}
On a class of bodies admitting a uniformly controlled finite
$C^{m,\alpha}$ atlas with a common tubular-neighborhood radius, and
satisfying
\[
 \Theta_{r_K}(K)\geq\theta>0,
\]
one has a uniform bound
\begin{equation}
 \|q_{r_K,K}\|_{C^{m-1,\alpha}(\partial K)}\leq C,
 \label{eq:q-higher-uniform-estimate}
\end{equation}
where $C$ depends only on $n,m,\alpha,r_K,\theta$ and the prescribed
geometric bounds.
\end{lemma}

\begin{proof}

Fix $x_0\in\partial K$, choose a $C^{m,\alpha}$ parametrization
\[
 \Psi:U\longrightarrow\partial K,
 \qquad
 \Psi(\xi_0)=x_0,
\]
and fix $U_0\Subset U$.  We prove that
\[
 Q(\xi):=q_{r_K,K}(\Psi(\xi))
 =
 \int_{\mathbb R^n}
 \kai_K(y)\kai_{\{|y-\Psi(\xi)|<r_K\}}\,dy
\]
belongs to $C^{m-1,\alpha}(U_0)$.

\medskip
\noindent
\textit{Step 1: mollification and the first derivative.}
We first establish the $C^1$ assertion by smoothing only the fixed
cutoff.  Let $\eta_\varepsilon$ be a standard mollifier supported in
$B_\varepsilon$, and set
\[
 \phi_\varepsilon=\kai_K*\eta_\varepsilon,
 \qquad
 Q_\varepsilon(\xi)
 :=
 \int_{B_{r_K}(\Psi(\xi))}\phi_\varepsilon(y)\,dy
 =
 \int_{B_{r_K}(0)}
 \phi_\varepsilon(\Psi(\xi)+z)\,dz.
\]
Then $Q_\varepsilon$ is smooth.  For
$\zeta\in\mathbb R^{n-1}$ and $x=\Psi(\xi)$, differentiation under the
integral sign, followed by the divergence theorem, gives
\begin{align*}
 DQ_\varepsilon(\xi)[\zeta]
 &=
 \int_{B_{r_K}(0)}
 \nabla\phi_\varepsilon(\Psi(\xi)+z)
 \cdot D\Psi(\xi)[\zeta]\,dz \notag\\
 &=
 \int_{\partial B_{r_K}(x)}
 \phi_\varepsilon(y)\,
 D\Psi(\xi)[\zeta]\cdot\frac{y-x}{r_K}\,
 d\mathcal H^{n-1}(y).
\end{align*}
Thus the differentiation is performed only for the smooth function
$\phi_\varepsilon$, never directly on the sharp cutoff $\kai_K$.

The expected limiting derivative is
\[
 G(\xi)[\zeta]
 :=
 \int_{K\cap\partial B_{r_K}(\Psi(\xi))}
 D\Psi(\xi)[\zeta]\cdot
 \frac{y-\Psi(\xi)}{r_K}\,d\mathcal H^{n-1}(y).
\]
We show that $Q_\varepsilon\to Q$ and
$DQ_\varepsilon\to G$ uniformly on $U_0$.  The first convergence is
immediate:
\begin{equation}
 \sup_{\xi\in U_0}|Q_\varepsilon(\xi)-Q(\xi)|
 \leq
 \|\phi_\varepsilon-\kai_K\|_{L^1(\mathbb R^n)}
 \longrightarrow0.
 \label{eq:parameter-intersection-function-convergence}
\end{equation}

It remains to compare the derivatives.  Here $\xi$ moves the centre
$\Psi(\xi)$, whereas $y$ is only the integration variable.  The
derivative error can occur only when the moving sphere passes near
$\partial K$, so we first control this part of the sphere.  On a fixed
tubular neighborhood of $\partial K$, let $\varphi$ be the signed-distance
defining function, negative in $K$.  Thus
\(
\nabla\varphi=\nu_K
\) on $\partial K$ and
\(
|\varphi|=\operatorname{dist}(\,\cdot\,,\partial K)
\) in this neighborhood.  Set
\[
 \mathcal I
 :=
 \left\{(\xi,y)\in\overline{U_0}\times\mathbb R^n:
 y\in\partial K\cap\partial B_{r_K}(\Psi(\xi))\right\}.
\]
At $(\xi,y)\in\mathcal I$, let
\[
 w=\frac{y-\Psi(\xi)}{r_K}.
\]
\[
 \begin{aligned}
 \nabla_{\partial B_{r_K}(\Psi(\xi))}\varphi(y)
 &=\nu_K(y)-\bigl(\nu_K(y)\cdot w\bigr)w,\\
 \left|\nabla_{\partial B_{r_K}(\Psi(\xi))}\varphi(y)\right|
 &=
 \left|\Pi_{T_y\partial K}w\right|
 \geq\Theta_{r_K}(K)>0.
 \end{aligned}
\]
Thus $\Theta_{r_K}(K)>0$ is exactly a uniform regular-value condition for
the restriction of $\varphi$ to every moving sphere at its zero set.
Writing a point of that sphere as $y=\Psi(\xi)+r_K\omega$ with
$\omega\in\mathbb S^{n-1}$, continuity and compactness give a number
$\delta>0$ such that
\[
 \left|\nabla_{\partial B_{r_K}(\Psi(\xi))}\varphi(y)\right|
 \geq\frac12\Theta_{r_K}(K)
 \quad\text{whenever}\quad
 \xi\in\overline{U_0},\quad
 y\in\partial B_{r_K}(\Psi(\xi)),\quad
 |\varphi(y)|\leq\delta.
\]
Indeed, a failure of this statement would yield, after passing to a
convergent subsequence, a point of $\mathcal I$ contradicting the
preceding lower bound.  A finite
coordinate cover of this compact family of spheres now gives
\[
 \sup_{\substack{\xi\in\overline{U_0}\\ |t|\leq\delta}}
 \mathcal H^{n-2}
 \bigl(\partial B_{r_K}(\Psi(\xi))\cap\{\varphi=t\}\bigr)
 \leq C_{U_0}.
\]
In fact, in each chart the preceding tangential-gradient lower bound makes
one coordinate derivative of $\varphi$ nonzero; the implicit-function
theorem writes $\{\varphi=t\}$ as a graph with uniformly bounded area.
For $2\varepsilon<\delta$, the coarea formula therefore gives, uniformly
for $\xi\in U_0$,
\begin{align}
 &\mathcal H^{n-1}
 \left(
 \left\{y\in\partial B_{r_K}(\Psi(\xi)):
 \operatorname{dist}(y,\partial K)\leq2\varepsilon\right\}
 \right) \notag\\
 &\quad\leq
 \int_{-2\varepsilon}^{2\varepsilon}
 \int_{\partial B_{r_K}(\Psi(\xi))\cap\{\varphi=t\}}
 \frac{1}{|\nabla_{\partial B_{r_K}(\Psi(\xi))}\varphi|}
 \,d\mathcal H^{n-2}\,dt
 \leq C_{U_0}\varepsilon.
 \label{eq:parameter-intersection-trace-estimate}
\end{align}
Since $\eta_\varepsilon$ is supported in $B_\varepsilon$, on the part
of $\partial B_{r_K}(\Psi(\xi))$ at distance greater than
$2\varepsilon$ from $\partial K$ one has
$\phi_\varepsilon=\kai_K$: both are one on the portion in $K$ and
zero on the portion outside $K$.  Thus the two constant regions cancel
in the preceding derivative formula, and hence
\[
 \begin{aligned}
 \bigl(DQ_\varepsilon(\xi)-G(\xi)\bigr)[\zeta]
 &={}
 \int_{\substack{y\in\partial B_{r_K}(\Psi(\xi))\\
 \operatorname{dist}(y,\partial K)\leq2\varepsilon}}
 \bigl(\phi_\varepsilon(y)-\kai_K(y)\bigr)
 D\Psi(\xi)[\zeta]\cdot
 \frac{y-\Psi(\xi)}{r_K}\,d\mathcal H^{n-1}(y).
 \end{aligned}
\]
For $|\zeta|=1$, the absolute value of the first factor in the
integrand is at most one, the vector
$(y-\Psi(\xi))/r_K$ has length one on the sphere, and
$D\Psi$ is bounded on $\overline{U_0}$.  Hence
\[
 \left|
 \bigl(DQ_\varepsilon(\xi)-G(\xi)\bigr)[\zeta]
 \right|
 \leq
 \|D\Psi\|_{L^\infty(U_0)}
 \mathcal H^{n-1}
 \left(
 \left\{y\in\partial B_{r_K}(\Psi(\xi)):
 \operatorname{dist}(y,\partial K)\leq2\varepsilon\right\}
 \right).
\]
Taking the supremum over $\xi\in U_0$ and using
\eqref{eq:parameter-intersection-trace-estimate}, with $C_{U_0}$
enlarged if necessary, we obtain
\[
 \sup_{\xi\in U_0}
 \left\|DQ_\varepsilon(\xi)-G(\xi)\right\|
 \leq C_{U_0}\varepsilon
 \longrightarrow0.
\]
The value assigned to $\kai_K$ on
$\partial K\cap\partial B_{r_K}(\Psi(\xi))$ is immaterial, since this
intersection has dimension $n-2$.  Thus
\begin{equation}
 DQ_\varepsilon\longrightarrow G
 \qquad\text{uniformly on }U_0.
 \label{eq:parameter-intersection-derivative-convergence}
\end{equation}
Since each $DQ_\varepsilon$ is continuous, so is $G$.  Combining
\eqref{eq:parameter-intersection-function-convergence} and
\eqref{eq:parameter-intersection-derivative-convergence}, the uniform
convergence theorem for $C^1$ functions gives $Q\in C^1(U_0)$ and
$DQ=G$; in particular,
\begin{equation}
 DQ(\xi)[\zeta]
 =\int_{K\cap\partial B_{r_K}(\Psi(\xi))}
 D\Psi(\xi)[\zeta]\cdot
 \frac{y-\Psi(\xi)}{r_K}\,d\mathcal H^{n-1}(y).
 \label{eq:parameter-intersection-first-derivative}
\end{equation}

\medskip
\noindent
\textit{Step 2: straightening the simultaneous-boundary region.}
Step~1 yields only $C^1$ regularity, and its argument cannot be
iterated: although the transition strip has measure $O(\varepsilon)$,
the higher derivatives of $\phi_\varepsilon$ need not be uniformly
bounded there.  For a fixed $\xi$, write
\[
 D_\xi:=K\cap B_{r_K}(\Psi(\xi)).
\]
Then
\[
 Q(\xi)=\int_{D_\xi}dy
 =\int_{\mathbb R^n}
 \kai_K(y)\,\kai_{\{|y-\Psi(\xi)|<r_K\}}\,dy.
\]
The first expression displays the moving geometric set, whereas the
second is an integral over the fixed ambient $y$-space.  We localize
the latter in the joint variables $(\xi,y)$; in particular, we do not
partition the moving set $D_\xi$ itself.  Fix
$(\xi_0,y_0)\in\mathcal I$, choose a
$C^{m,\alpha}$ defining function $\varphi$ near $y_0$ with
\[
 K=\{\varphi<0\},
 \qquad
 \nabla\varphi=\nu_K\quad\text{on }\partial K,
\]
and write
\[
 F(\xi,y)=|y-\Psi(\xi)|^2-r_K^2.
\]
Near $(\xi_0,y_0)$, the integrand in $Q$ is
\[
 \kai_{\{\varphi(y)<0\}}\kai_{\{F(\xi,y)<0\}}.
\]
On the common zero set of $\varphi$ and $F$,
\[
 \left|\nabla_y\varphi\wedge\nabla_yF\right|
 =
 2r_K\left|\Pi_{T_y\partial K}
 \frac{y-\Psi(\xi)}{r_K}\right|
 \geq2r_K\Theta_{r_K}(K)>0.
 \label{eq:parameter-intersection-jacobian-mollified-proof}
\]
Thus $\nabla_y\varphi(y)$ and $\nabla_yF(\xi,y)$ are linearly
independent.  Equivalently, $\partial K$ and
$\partial B_{r_K}(\Psi(\xi))$ meet transversely at their common
points.  Choose $n-2$ tangential coordinates $z=z(y)$ so that
$D_y(\varphi,F,z)$ is nonsingular at $(\xi_0,y_0)$.  The
parameter-dependent inverse-function theorem then yields a
$C^{m,\alpha}$ change of variables
\[
 (\xi,y)\longmapsto
 (\xi,s,t,z)
 =
 \bigl(\xi,\varphi(y),F(\xi,y),z(y)\bigr)
\]
with inverse $y=Y(\xi,s,t,z)$.  In particular, the two membership
conditions become the fixed inequalities
\[
s<0,\qquad t<0.
\]

By compactness of $\mathcal I$, finitely many such coordinate
neighborhoods cover it.  On a neighborhood of the compact active set
\[
 \mathcal A
 :=
 \left\{(\xi,y)\in\overline{U_0}\times\mathbb R^n:
 y\in K\cap\overline{B_{r_K}(\Psi(\xi))}\right\},
\]
choose a smooth partition of unity
$\vartheta_0,\vartheta_1,\ldots,\vartheta_N$, with
$\vartheta_j$ for $j\geq1$ supported in one of these coordinate
neighborhoods and $\vartheta_0$ supported away from $\mathcal I$.
For every fixed $\xi\in U_0$, this gives the exact weighted
decomposition
\[
 Q(\xi)
 =
 \sum_{j=0}^{N}
 \int_{\mathbb R^n}
 \vartheta_j(\xi,y)\,
 \kai_K(y)\,\kai_{\{F(\xi,y)<0\}}\,dy.
\]
The supports of the $\vartheta_j$ may overlap, but their total weight
is one at every pair $(\xi,y)$ for which the displayed integrand can be
nonzero.  Thus this is a weighted decomposition of the integral, not a
set-theoretic decomposition of $D_\xi$ into disjoint pieces.  The
terms with $j\geq1$ are localized near the simultaneous-boundary set
$\mathcal I$, where both conditions vary and the preceding coordinate
change is used; the $j=0$ term is supported away from $\mathcal I$ and
is treated in Step~3.  For $1\leq j\leq N$, let $Q_j$ denote the
corresponding integral in the preceding decomposition, and let $Y_j$
be the inverse of the coordinate map in the neighborhood containing
$\operatorname{supp}\vartheta_j$.  For fixed $\xi$,
$\vartheta_j(\xi,\cdot)$ vanishes outside this neighborhood, where the
preceding map $y\mapsto(s,t,z)$ is a diffeomorphism.  Thus the ordinary
change-of-variables formula applies to $Q_j$: the two indicators
become $\kai_{\{s<0\}}\kai_{\{t<0\}}$, and the Jacobian is that of
$Y_j$.  Hence
\[
 Q_j(\xi)
 =
 \int_{D_j}
 a_j(\xi,s,t,z)\,ds\,dt\,dz,
\]
where $D_j$ is a fixed coordinate box intersected with
$\{s<0,t<0\}$; the cutoff $\vartheta_j$ makes the integrand vanish
outside this box.  Moreover,
\[
 a_j(\xi,s,t,z)
 =
 \vartheta_j\bigl(\xi,Y_j(\xi,s,t,z)\bigr)
 \left|\det D_{(s,t,z)}Y_j(\xi,s,t,z)\right|
 \in C^{m-1,\alpha}.
\]
The domain in the preceding formula does not depend on $\xi$.  Hence differentiation under the integral
sign gives
\[
 Q_j\in C^{m-1,\alpha}.
\]
The single lost derivative comes from the Jacobian of $Y_j$.  Thus all
terms localized near $\mathcal I$ have the required regularity; it
remains to consider the complementary terms, where at most one
boundary condition is active.

\medskip
\noindent
\textit{Step 3: the complement of the simultaneous-incidence set.}
It remains to treat the contribution localized by $\vartheta_0$.  Since
$\operatorname{supp}\vartheta_0\cap\mathcal I=\varnothing$, for each
$(\xi_0,y_0)\in\operatorname{supp}\vartheta_0$ the quantities
$\varphi(y_0)$ and $F(\xi_0,y_0)$ cannot both vanish.  By compactness,
after a finite refinement with a smooth partition of unity it suffices
to consider
\begin{equation}
 Q_\ell(\xi)
 =
 \int_{\mathbb R^n}
 \mathbf 1_\ell(\xi,y)\,
 \kai_{\{\varphi(y)<0\}}\kai_{\{F(\xi,y)<0\}}\,dy,
 \label{eq:parameter-away-localized}
\end{equation}
where on $\operatorname{supp}\mathbf 1_\ell$ at least one of the two
indicator functions is constant.  There are three cases.

\noindent\emph{Case 1: the support meets only the moving spherical
boundary.}  Suppose that $\kai_{\{\varphi<0\}}$ is constant on
$\operatorname{supp}\mathbf 1_\ell$.  The zero case gives
$Q_\ell\equiv0$; otherwise $\varphi<0$ throughout the support and,
with $y=\Psi(\xi)+z$,
\begin{equation}
 Q_\ell(\xi)
 =
 \int_{B_{r_K}(0)}b_\ell(\xi,z)\,dz,
 \qquad
 b_\ell(\xi,z):=\mathbf 1_\ell(\xi,\Psi(\xi)+z).
 \label{eq:parameter-away-ball}
\end{equation}
Since $b_\ell\in C^{m,\alpha}$ and the domain is fixed,
$Q_\ell\in C^{m,\alpha}(U_0)$.

\noindent\emph{Case 2: the support meets only $\partial K$.}  Suppose
that $\kai_{\{F<0\}}$ is constant on
$\operatorname{supp}\mathbf 1_\ell$.  Again the zero case is trivial;
otherwise $F<0$ throughout the support and
\begin{equation}
 Q_\ell(\xi)
 =
 \int_{\mathbb R^n}
 \mathbf 1_\ell(\xi,y)\kai_{\{\varphi(y)<0\}}\,dy.
 \label{eq:parameter-away-K}
\end{equation}
Choose fixed boundary coordinates $y=Y(s,z)$ with $s=\varphi(y)$.
Then
\begin{equation}
 Q_\ell(\xi)
 =
 \int_{D_\ell\cap\{s<0\}}b_\ell(\xi,s,z)\,ds\,dz,
 \label{eq:parameter-away-K-coordinates}
\end{equation}
where
\[
 b_\ell(\xi,s,z)
 =
 \mathbf 1_\ell(\xi,Y(s,z))
 \left|\det D_{(s,z)}Y(s,z)\right|.
\]
Since $Y\in C^{m,\alpha}$, the density is in $C^{m-1,\alpha}$ and the
domain is independent of $\xi$.

\noindent\emph{Case 3: the support meets neither boundary.}  If both
indicators are constant on $\operatorname{supp}\mathbf 1_\ell$, then
either $Q_\ell\equiv0$, or
\begin{equation}
 Q_\ell(\xi)=\int_D b_\ell(\xi,y)\,dy
 \label{eq:parameter-away-interior}
\end{equation}
for a fixed bounded domain $D$ and some
$b_\ell\in C^{m,\alpha}$.  Hence each contribution away from
$\mathcal I$ belongs to $C^{m-1,\alpha}(U_0)$.

Summing the finite partition gives
\[
 Q\in C^{m-1,\alpha}(U_0).
\]
Since $x_0$ and $U_0$ were arbitrary,
\eqref{eq:q-higher-regularity} follows.  The same finite-cover
argument, together with the uniform atlas, tubular-radius, and
transversality bounds in the statement, gives the asserted uniform
estimate.
\end{proof}

Combining Lemma~\ref{lem:parameter-intersection} with the usual
inductive bootstrap and the regularity theory for uniformly elliptic
equations in gnomonic coordinates yields the following higher-regularity
consequence.

\begin{theorem}[Higher regularity]
\label{thm:higher-regularity}
Let $k\geq0$ be an integer and $0<\alpha<1$. Suppose that
\[
 \mathcal M_{\rho,K}=f\,d\sigma,
 \qquad
 f\in C^{k,\alpha}(\mathbb S^{n-1}),
 \qquad
 f>0.
\]
If
\[
 r_K<\operatorname{diam}K,
 \qquad
 \Theta_{r_K}(K)>0,
\]
then
\begin{equation}
 h_K\in C^{k+2,\alpha}(\mathbb S^{n-1}).
 \label{eq:higher-reg-conclusion}
\end{equation}
If $f\in C^\infty$, then $K\in C^\infty_+$.
\end{theorem}

\begin{remark}[The saturated regime]
\label{rem:saturated-regime}
The preceding analysis is unnecessary when
$r_K\geq\operatorname{diam}K$. The ball
$B_{r_K}(x)$ covers $K$ up to a null set for every
$x\in\partial K$, and the covariogram vanishes on
$\partial B_{r_K}$.  Consequently,
\[
 q_{r_K,K}(x)=V,\qquad
 \beta_{\rho,K}=0,\qquad
 w_K(x)=2V,
 \qquad
 \mathcal M_{\rho,K}=2VS_K.
\]
Thus, if $\mathcal M_{\rho,K}=f\,d\sigma$, the equation is simply
\[
 S_K=\frac{f}{2V}\,d\sigma.
\]
It is therefore the classical Minkowski problem; in particular,
positive $C^{k,\alpha}$ data give $K\in C^{k+2,\alpha}_+$, smooth data
give $K\in C^\infty_+$, and no transversality assumption is required.
Moreover, since
$\operatorname{diam}K\geq2R_K$ and $r_K=\rho R_K$, this regime can
occur only for $\rho\geq2$.  If $\rho=2$, equality holds in the
isodiametric inequality and $K$ is a ball; for $\rho>2$, nonspherical
convex bodies may also occur.
\end{remark}

\section*{Acknowledgement}
The author wishes to express her sincere gratitude to her supervisor, Prof. Yong Huang, for the many fruitful discussions, guidance, and assistance throughout the completion of this paper. She would also like to extend her sincere thanks to Prof. Gaoyong Zhang for his valuable comments and suggestions on this paper.

\bibliographystyle{amsplain}
\providecommand{\bysame}{\leavevmode\hbox to3em{\hrulefill}\thinspace}
\providecommand{\MR}{\relax\ifhmode\unskip\space\fi MR }
\providecommand{\MRhref}[2]{%
  \href{http://www.ams.org/mathscinet-getitem?mr=#1}{#2}
}
\providecommand{\href}[2]{#2}

\end{document}